\DeclareFontFamily{OT1}{rsfs}{}
\DeclareFontShape{OT1}{rsfs}{n}{it}{<-> rsfs10}{}
\DeclareMathAlphabet{\mathscr}{OT1}{rsfs}{n}{it}

\documentclass[12pt,amssymb]{amsart}
\usepackage{amsmath}
\usepackage{amssymb}
\usepackage{mathabx}
\DeclareFontFamily{OT1}{rsfs}{}
\DeclareFontShape{OT1}{rsfs}{n}{it}{<-> rsfs10}{}
\DeclareMathAlphabet{\mathscr}{OT1}{rsfs}{n}{it}

\newcommand{\Z}{{\mathbb Z}}

\newcommand{\Q}{{\mathbb Q}}

\newcommand{\R}{{\mathbb R}}

\newcommand{\A}{\mathbb{A}}

\newcommand{\cP}{\mathscr{P}}

\newcommand{\md}{{\rm mod}\ }

\newcommand{\Ga}{\mathrm{Gal}}
\newtheorem{thm}{Theorem}[section]
\newtheorem{lemma}[thm]{Lemma}
\newtheorem{prop}[thm]{Proposition}
\newtheorem{cor}[thm]{Corollary}
\theoremstyle{remark}
\newtheorem{remark}[thm]{Remark}
\theoremstyle{defin}
\newtheorem*{defin}{Definition}

\newtheorem*{thmA}{Theorem A}
\newtheorem*{thmB}{Theorem B}
\newtheorem*{thmC}{Theorem C}
\newtheorem*{thmD}{Theorem D}
\newtheorem*{theor}{Theorem} 

\newcommand{\cG}{\mathscr{G}}

\numberwithin{equation}{section}

\usepackage[all,cmtip]{xy}
.240pk scaled 1200 .240pk
\begin{document}

\title[Congruence kernel]{On the congruence kernel for simple algebraic groups}

\author[Prasad]{Gopal Prasad}
\author[Rapinchuk]{Andrei S. Rapinchuk}

\begin{abstract}
This paper contains several results about the structure of the congruence kernel $C^{(S)}(G)$ of an absolutely
almost simple simply connected algebraic group $G$ over a global field $K$ with respect to a set of places $S$ of $K$.
In particular, we show that $C^{(S)}(G)$ is always trivial if $S$ contains a generalized arithmetic progression. We also
give a criterion for the centrality of $C^{(S)}(G)$ in the general situation in terms of the existence of commuting lifts
of the groups $G(K_v)$ for $v \notin S$ in the $S$-arithmetic completion $\widehat{G}^{(S)}$. This result enables one to give simple proofs of the
centrality in a number of cases. Finally, we show that if $K$ is a number field and $G$ is $K$-isotropic then
$C^{(S)}(G)$ as a normal subgroup of $\widehat{G}^{(S)}$ is almost generated by a single element.
\end{abstract}

\address{Department of Mathematics, University of Michigan, Ann
Arbor, MI 48109}

\email{gprasad@umich.edu}

\address{Department of Mathematics, University of Virginia,
Charlottesville, VA 22904}

\email{asr3x@virginia.edu}

\maketitle

\vskip3mm

\hfill {\it To V.P.~Platonov on his 75th birthday}

\vskip3mm

\section{Introduction}\label{S:Intro}

Let $G$ be an absolutely almost simple simply connected algebraic group defined over a global field $K$, and let $S$ be a nonempty
subset of the set $V^K$ of all places of $K$ containing the set $V_{\infty}^K$ of archimedean places. We fix a $K$-embedding $G \hookrightarrow \mathrm{SL}_n$
and define $$G(\mathcal{O}(S)) = G(K) \, \cap \, \mathrm{SL}_n(\mathcal{O}(S)),$$ where $\mathcal{O}(S)$  is the ring of $S$-integers in $K$.
One then introduces two topologies, $\tau_a$ and $\tau_c$, on the group of $K$-rational points
$G(K)$, called the $S$-{\it arithemetic topology} and the $S$-{\it congruence topology}, respectively, by taking for a fundamental system of neighborhoods of the
identity all normal subgroups of finite index $N \subset G(\mathcal{O}(S))$ for $\tau_a$, and the congruence subgroups $G(\mathcal{O}(S) , \mathfrak{a}) =  G(K) \, \cap
\, \mathrm{SL}_n(\mathcal{O}(S) , \mathfrak{a})$ corresponding to nonzero ideals $\mathfrak{a}$ of $\mathcal{O}(S)$\footnote{As usual, $\mathrm{SL}_n(\mathcal{O}(S) , \mathfrak{a}) = \{ A \in \mathrm{SL}_n(\mathcal{O}(S)) \, \vert \, A \equiv I_n(\mathrm{mod} \: \mathfrak{a}) \}$.} for $\tau_c$. One shows that these topologies in fact do not depend on the choice of the original $K$-embedding of $G$ into $\mathrm{SL}_n$, and furthermore, the group $G(K)$ admits completions with respect to both $\tau_a$ and $\tau_c$. These completions will be denoted $\widehat{G}^{(S)}$ and $\overline{G}^{(S)}$, and called respectively the $S$-arithmetic and the $S$-congruence completions. As the topology $\tau_a$ is finer than $\tau_c$, there is a natural continuous homomorphism $\pi^{(S)} \colon \widehat{G}^{(S)} \to \overline{G}^{(S)}$, which turns out to be surjective. Its kernel $C^{(S)}(G)$ is called the $S$-\emph{congruence kernel}. Clearly, $C^{(S)}(G)$ is trivial if and only if every normal subgroup $N \subset G(\mathcal{O}(S))$ contains a congruence subgroup $G(\mathcal{O}(S) , \mathfrak{a})$ for some $\mathfrak{a}$, which means that we have an affirmative answer to the classical congruence subgroup problem for the group $G(\mathcal{O}(S))$. In general,  $C^{(S)}(G)$ measures the deviation from the affirmative answer, so by the congruence subgroup problem in a broader sense one means the task of  computing $C^{(S)}(G)$. (In the sequel, we will omit the superscript $(S)$ if this may not lead to a confusion.)

The investigation of the congruence subgroup problem has two aspects: the first  is to prove that in certain cases $C^{(S)}(G)$ is finite and then determine it precisely, and the other is to understand the structure of $C^{(S)}(G)$ in the cases where it is infinite. We recall that the expected conditions for $C^{(S)}(G)$ to be finite/infinite are given in
the following conjecture of Serre \cite{Serre1}:

\vskip2mm

\begin{center}

{\it $C^{(S)}(G)$ should be finite if $\mathrm{rk}_S \: G := \sum_{v \in S} \mathrm{rk}_{K_v} \: G$ is $\geqslant 2$ and $G$ is $K_v$-isotropic

for all $v \in S \setminus V_{\infty}^K$, and $C^{(S)}(G)$ should be infinite if $\mathrm{rk}_S \: G = 1$.}

\end{center}

\noindent (In the sequel, we will always assume that  $\mathrm{rk}_S \: G > 0$ as otherwise the group $G(\mathcal{O}(S))$ is finite, hence $C^{(S)}(G)$ is trivial.) The results of the current paper contribute to both aspects of the congruence subgroup problem. (We refer the reader to the surveys \cite{PR-Milnor} and \cite{Ra3} and references therein for information about the very impressive body of work in this area.)

To give the precise formulations, we need to recall the statement of the Margulis-Platonov conjecture (MP) for the group $G(K)$:

\vskip2mm

\begin{center}

\parbox[t]{15cm}{{\it Let $\mathcal{A} = \{ v \in V^K \setminus V_{\infty}^K \, \vert \, \mathrm{rk}_{K_v} \: G = 0 \}$ be the set of nonarchimedean places of $K$ where  $G$ is anisotropic, let  $G_{\mathcal{A}} = \prod_{v \in \mathcal{A}} G(K_v)$, and let $\delta \colon G(K) \to G_{\mathcal{A}}$ be the diagonal map. Then for any noncentral normal subgroup $N$ of $G(K)$, there is an open normal subgroup $U$ of $G_{\mathcal{A}}$ such that $N = \delta^{-1}(U)$; in particular, if $\mathcal{A} = \varnothing$ (which is always the case if $G$ is not of type $\textsf{A}_n$) then $G(K)$ does not contain any proper noncentral normal subgroups.}}

\end{center}

\vskip2mm

\noindent We note that (MP) has been established in all cases where  $G$ is $K$-isotropic (see \cite{Gille}) and also in many cases where $G$ is $K$-anisotropic (see \cite[Ch.\,IX]{PlRa} and \cite[Appendix A]{RS}). Throughout this paper, we reserve the notation $\mathcal{A} = \mathcal{A}(G)$ for the set of anisotropic nonarchimedean places of $G$ and {\it  assume that $({\mathrm{MP}})$ holds for $G(K)$ and that $\mathcal{A} \cap S = \varnothing$}. (As shown in \cite[\S 6]{Ra1}, the general case can be reduced to the case where $\mathcal{A} \cap S = \varnothing$, but if the latter fails then $C^{(S)}(G)$ is always infinite.) It is known that $C^{(S)}(G)$ is finite if and only if it is \emph{central} (i.e., is contained in the center of $\widehat{G}^{(S)}$ ), in which case it is isomorphic to the Pontrjagin dual of the \emph{metaplectic kernel} $M(S , G)$ (cf.\,\cite[\S 3]{PR-Milnor}).  Since the metaplectic kernel has completely been determined in \cite{PR1} in all cases relevant for the congruence subgroup problem  (for example, we know that $M(S,G)$ is trivial if $S$ is infinite), the first of the two aspects of the congruence subgroup problem we mentioned above reduces to proving that $C^{(S)}(G)$ is central in the expected cases. Our first basic result (Theorem 4.3) is the following. This result was proved in \cite[\S 9]{PR1} in the case where $K$ is a number field.

\begin{theor}
Let $G$ be an absolutely almost simple simply connected algebraic group over a global field $K$, and assume that $({\mathrm{MP}})$ holds for $G(K)$. Then for
any finite set $V$ of nonarchimedean places of $K$ that contains the set $\mathcal{A}$ of anisotropic places and $S = V^K \setminus V$, the congruence kernel $C^{(S)}(G)$ is central and hence trivial.
\end{theor}

This result  will be used to prove the following theorem which provides a new, and particularly effective, criterion for the centrality of $C^{(S)}(G)$. We observe that since $\mathrm{rk}_S \: G > 0$, it follows from the strong approximation property (cf.\,\cite{Ma-SA}, \cite{Pl-SA}, \cite{Pr-SA}; see also \cite{Ra-SA} for a recent survey) that the congruence completion $\overline{G}^{(S)}$ can be naturally identified with the group of $S$-adeles $G(\mathbb{A}(S))$, which enables us to view the group $G(K_v)$ for any $v \in V^K \setminus S$ as a subgroup of $\overline{G}^{(S)}$. As above, we let $\pi \colon \widehat{G} \to \overline{G}$ denote the natural continuous homomorphism (we suppress the superscript $(S)$).

\begin{thmA}
Let $G$ be an absolutely almost simple
simply connected algebraic group over a~global field $K$,
and let $S$ be \emph{any} subset of $V^K\setminus \mathcal{A}$ containing $V^K_{\infty}$.
Assume that for every $v \notin S$, there is  a subgroup $\cG_v$ of
$\widehat{G}$ so that the following conditions are satisfied:

\vskip2mm

$(i)$ \parbox[t]{12cm} {$\pi(\cG_v) = G(K_v)$ \ for all \ $v \notin S$;}

\vskip1mm

 $(ii)$ \parbox[t]{12cm}{$\cG_{v_1}$ and $\cG_{v_2}$ commute elementwise for all $v_1 ,
v_2 \notin S,$ $v_1 \neq v_2;$}

\vskip1mm

$(iii)$ \parbox[t]{12cm}{the subgroup generated by the $\cG_v$, for $v \notin S,$ is
dense in $\widehat{G}.$}

\vskip2mm

\noindent Then $C^{(S)}(G)$ is central in $\widehat{G}.$
\end{thmA}

(We note that  this theorem was already stated in \cite[Theorem 7]{Ra-Bass} and that Proposition \ref{P:Com-lifts2} contains a somewhat more general result which is sometimes  useful.) We will show in \S \ref{S:Appl} how Theorem A can be used to establish the centrality of the congruence kernel for $G = \mathrm{SL}_{n}$, $n \geqslant 3$ and $G = \mathrm{SL}_2$ when $\mathrm{rk}_S \: G \geqslant 2$ (i.e., when the group of units $\mathcal{O}(S)^{\times}$ is infinite) - see Examples 4.6 and 4.7.

To formulate our next result, we need to recall the definition of a~\emph{generalized arithmetic progression}.
For a global field $K$, we let $V^K_f$ denote the set of all nonarchimedean places of $K$ (i.e. $V^K_f = V^K \setminus V^K_{\infty}$).
Now, let $F/K$ be a Galois extension (not necessarily abelian) with  Galois group $\mathscr{G}=
\mathrm{Gal}(F/K)$. Given $v \in V^K_f$ which is unramified in $F$, for every extension $w \vert v$ one defines
the Frobenius automorphism $\mathrm{Fr}_{F/K}(w \vert v) \in \mathscr{G}$; recall that $\mathrm{Fr}_{F/K}(w \vert v)$
for \emph{all} extensions $w \vert v$ fill a conjugacy class of $\mathscr{G}$ (cf.\,\cite[Ch.\,VII]{ANT}). Now,
fix a conjugacy class $\mathscr{C}$ of $\mathscr{G}.$

\begin{defin}
A {\it generalized arithmetic progression}
$\cP(F/K , \mathscr{C})$ is the set of all
$v \in V^K_f$ such that $v$ is unramified in $F/K$ and for some
(equivalently, any) extension $w \vert v$, the Frobenius
automorphism $\mathrm{Fr}_{F/K}(w \vert v)$ is in the conjugacy class $\mathscr{C}$.
\end{defin}

\vskip1mm

We can now formulate our next result.
\begin{thmB}\label{T:B}
Let $G$ be an absolutely almost simple simply connected algebraic group over a~global field $K$, and let $L$ be the minimal Galois extension of $K$ over
which $G$ becomes an inner form of the split group. Assume that a subset $S \subset V^K$ is disjoint from $\mathcal{A}$, contains $V^K_{\infty}$ and also contains all but finitely many places belonging to a generalized arithmetic progression $\cP(F/K , \mathscr{C})$ such that $\sigma \vert (F \cap L) = \mathrm{id}_{F \cap L}$ for some (equivalently, any) $\sigma \in \mathscr{C}$ (which is automatically true if $G$ is an inner form of the split group over $K$). Then $C^{(S)}(G)$ is central, and hence in fact trivial.
\end{thmB}

\vskip1mm

In \S \ref{S:Adelic} we will give a new proof of Lubotzky's conjecture on the congruence subgroup property for arithmetic groups with adelic profinite completion. A profinite group $\Delta$ is called \emph{adelic} if for some $n \geqslant 1$,
there exists a continuous embedding
$\iota \colon \Delta \hookrightarrow \mathrm{GL}_n(\widehat{\mathbb Z})$, where
$\displaystyle \widehat{\mathbb Z} =
\prod_{q\ {\rm prime}} {\mathbb Z}_q.$ It was conjectured by A.\:Lubotzky
that if for $\Gamma = G(\mathcal{O}(S))$ the profinite completion $\widehat{\Gamma}$ is adelic
then $\Gamma$ has the \emph{congruence subgroup property} (CSP), i.e. the congruence kernel $C^{(S)}(G)$ is finite.
This was proved by Platonov and Sury in \cite{PS} using some rather technical
constructions developed earlier in  \cite{PR2} to establish (CSP)
for arithmetic groups with bounded generation.  Subsequently, Liebeck and
Pyber \cite{LPy} showed that {\it any finitely generated} subgroup of
${\rm{GL}}_n(\widehat{\mathbb Z})$ has bounded generation, which allows one to prove
Lubotzky's conjecture by directly quoting the results of \cite{Lu1} and
\cite{PR2} on (CSP) for arithmetic groups with bounded generation.
We note  that it is essential in both \cite{LPy} and \cite{PS}
that $\Gamma$ be finitely generated (i.e., $S$ be finite). We give a rather short
proof of Lubotzky's conjecture that does not rely on finite generation (hence is applicable
even when $S$ is infinite).
\begin{thmC}
Let $G$ be an absolutely almost
simple simply connected algebraic group defined over a number field $K,$
$S \subset V^K \setminus \mathcal{A}$ be a subset containing $V^K_{\infty}$, and $\Gamma =
G({\mathcal O}(S)).$
If the profinite completion $\widehat{\Gamma}$ is adelic
then $C^{(S)}(G)$ is central, hence finite.
\end{thmC}

Our last result addresses the second aspect of the congruence subgroup problem, viz. the structure of the congruence
kernel $C = C^{(S)}(G)$ when it is infinite. It is known (cf. Proposition \ref{P:Prop-F}) that in this case the group $C$ is not finitely generated; for its precise structure
in certain cases see \cite{Lu0},  \cite{MSZ}, \cite{Me},  \cite{Z1}, \cite{Z2}. Lubotzky \cite{Lu3} showed however that the congruence kernel $C$ is always finitely generated as a \emph{normal} subgroup of $\widehat{\Gamma}$.  We will prove that when $K$ is a number field and $G$ is $K$-isotropic, $C$  as a normal subgroup of $\widehat{G}$ is almost generated by \emph{one} element (this result was announced more than 10 years ago and is mentioned in \cite{Lu3}, but its proof given below appears in print for the first time).
\begin{thmD}
Let $G$ be an absolutely almost simple simply connected algebraic group over a~number field $K$ with $\mathrm{rk}_K \: G = 1$. Then there exists
$c$ in $C = C^{(S)}(G)$ such that if $D$ is the closed normal subgroup of $\widehat{G}$ generated by $c$ then the quotient $C/D$ is a quotient of the metaplectic
kernel $M(S , G)$; in particular, it is a finite cyclic group.
\end{thmD}
(We note that if $G$ is $K$-isotropic and $\mathrm{rk}_S \: G \geqslant 2$ then $C^{(S)}(G)$ is known to be central (Raghunathan \cite{Ra1}, \cite{Ra2}) hence isomorphic to a quotient $M(S , G)$, so the theorem trivially holds with $c = 1$. Thus, the core case in the theorem is where $\mathrm{rk}_S \: G = 1$.)

\vskip5mm

\section{Preliminaries on the congruence kernel}\label{S:Prelim}

Let $G\,(\hookrightarrow {\rm{SL}}_n)$ be an absolutely almost simple simply connected algebraic group over a global field $K$, $\mathcal{A}$ the finite set nonarchimedean places of $K$ where $G$ is anisotropic and $S \subset V^K$ be a nonempty subset containing $V^K_{\infty}$ when
$K$ is a number field and such that $\mathcal{A}\cap S=\varnothing$ and $\mathrm{rk}_S\: G > 0$.  Let $\Gamma = G({\mathcal{O}}(S))$. The discussion in \S \ref{S:Intro} leads to the following exact sequence of topological
groups for the congruence kernel $C = C^{(S)}(G)$:
\begin{equation}\tag{C}\label{E:C}
1 \to C \longrightarrow \widehat{G} \stackrel{\pi}{\longrightarrow} \overline{G} \to 1
\end{equation}
(we omit the superscript $(S)$ whenever possible). It is an immediate consequence of the definitions that (\ref{E:C}) splits over the group of $K$-rational points $G(K)$ in
the category of abstract groups (with the image of this splitting being dense in $\widehat{G}$). Furthermore, as we already pointed out in \S \ref{S:Intro}, it follows from
the Strong Approximation Property that the $S$-congruence completion $\overline{G}$ can be naturally identified with the group of $S$-adeles $G(\mathbb{A}(S))$. We now
recall the following {\it universal property} of (\ref{E:C}).
\begin{prop}\label{P:Univ}
Let
$$
1 \to D \longrightarrow E \stackrel{\rho}{\longrightarrow} G(\A(S))
\to 1
$$
be an exact sequence of locally compact topological groups with $D$
a profinite group. Assume that there exists a splitting $\varphi \colon
G(K) \to E$ for $\rho$ over $G(K)$ whose image is dense in $E.$ Then there
exists a continuous surjective homomorphism $\nu \colon C \to D.$ In
particular, if $C$ is trivial then so is $D.$
\end{prop}
\begin{proof}
Since the closure $\overline{\Gamma}$ of $\Gamma$ in $\overline{G}$
is an open profinite subgroup, and the group $D$ is also profinite,
we see that $\Omega : = \rho^{-1}(\overline{\Gamma})$ is an open
profinite subgroup of $E.$ Then
$$
\varphi(G(K)) \cap \Omega = \varphi(G(K) \cap \overline{\Gamma}) =
\varphi(\Gamma)
$$
is a dense subgroup of $\Omega.$ By the universal property of the
profinite completion, there exists a continuous surjective
homomorphism $\widehat{\varphi} \colon \widehat{\Gamma} \to \Omega$
which coincides with $\varphi$ on $\Gamma.$ As $\varphi$ is a
section for $\rho$ over $G(K),$ the composition $\rho \circ
\widehat{\varphi} \colon \widehat{\Gamma} \to \overline{\Gamma}$
restricts to the identity map on $\Gamma,$ and therefore coincides
with $\pi$  on $\widehat{\Gamma}.$ Since $\widehat{\varphi} \colon
\widehat{\Gamma} \to \Omega$ is surjective, we now conclude that
$$
D = \widehat{\varphi}(\widehat{\varphi}^{-1}(D)) = \widehat{\varphi}(C),
$$
so $\nu := \widehat{\varphi} \vert C$ is as required.
\end{proof}

The goal of this section is to develop some techniques that will be used later to establish the centrality
of $C$ in certain situations. For further use, it is
convenient to deal not only with the extension (\ref{E:C})
itself, but  also with its  quotients. So, let $D \subset
\widehat{G}$ be a closed normal subgroup contained in $C.$ Consider
the quotient of (\ref{E:C}) by $D$:
\begin{equation}\tag{F}\label{E:F}
1 \to F = C/D \longrightarrow H = \widehat{G}/D
\stackrel{\theta}{\longrightarrow}  G(\A(S)) \to 1.
\end{equation}
We note that just like (\ref{E:C}), the sequence (\ref{E:F}) splits over the group $G(K)$, and the map $\theta$ is open and closed.
The following set of places plays an important role in examining when (\ref{E:F}) is a central extension:
$$
\mathfrak{Z}(F) = \{ v \in V^K \setminus (S \cup \mathcal{A}) \: \vert \: \theta(Z_H(F)) \supset G(K_v) \}
$$
where $Z_H(F)$ denotes the centralizer of $F$ in $H$, the set $\mathcal{A}$ consists of those nonarchimedean
$v \in V^K$ for which $G$ is $K_v$-anisotropic, and $G(K_v)$ is naturally identified with a subgroup of $\overline{G} =
G(\mathbb{A}(S))$. We will write $\mathfrak{Z}$ for $\mathfrak{Z}(C)$ if this will not lead to a confusion.
We note that Proposition \ref{P:Univ} is independent of the Margulis-Platonov conjecture (MP), but
in the rest of this section we do invoke our standing assumption that (MP) holds for $G(K)$ and $S \cap \mathcal{A} = \varnothing$.

We begin with a couple of results that give sufficient conditions for a place $v \in V^K \setminus (S \cup \mathcal{A})$
to belong to $\mathfrak{Z}(F)$.
\begin{prop}\label{P:v-in-Z1}
Let $v \in V^K \setminus (S \cup \mathcal{A}).$ Assume there exists
a noncentral $a \in G(K_v)$ such that $a \in \theta(Z_H(x))$ for
every $x \in F.$ Then $v \in \mathfrak{Z}(F).$
\end{prop}
\begin{proof}
Let $\mathcal{F}$ be the set of conjugacy classes in $F.$ The
conjugation action of $H$ on $F$ gives rise to a group homomorphism
$\overline{G} \stackrel{\tau}{\longrightarrow}
\mathrm{Perm}(\mathcal{F})$ to the group of permutations of
$\mathcal{F}.$ Our assumption means that $\tau(a) =
\mathrm{id}_{\mathcal{F}}.$ Since $G(K_v)$ does not have proper
noncentral normal subgroups (cf. \cite{Pl-SA}, and also \cite{Gille}, \cite{Pr-Rag}, \cite{Tits-Bour}), this implies
that $\tau(G(K_v)) = \{ \mathrm{id}_{\mathcal{F}} \}.$ In
particular, any open normal subgroup $W \subset F$ is normalized by
$H_v := \theta^{-1}(G(K_v)),$ so the latter acts on the finite group
$F/W.$ Let $\lambda_W \colon H_v \to \mathrm{Aut}(F/W)$ be the corresponding group
homomorphism, and set $\mathcal{L}_W = \mathrm{Ker} \: \lambda_W$. We need the following
lemma.
\begin{lemma}\label{L:A1}
Let $\kappa \colon \mathcal{H} \to \mathcal{G}$ and $\lambda \colon \mathcal{H} \to \mathcal{M}$ be two continuous homomorphisms of locally compact topological groups with
the kernels $\mathcal{K}$ and $\mathcal{L}$, respectively. Assume that

\medskip

\noindent {\rm (1)} \parbox[t]{15cm}{$\kappa$ is closed and surjective, $\mathcal{K}$ is compact, and $\mathcal{G}$ does not have
proper closed normal subgroups of finite index;}

\vskip1mm

\noindent {\rm (2)} $\mathcal{M}$ is profinite.

\medskip

\noindent Then $\mathcal{H} = \mathcal{K} \mathcal{L}$, or equivalently $\kappa(\mathcal{L}) = \mathcal{G}$.
\end{lemma}
\begin{proof}
Assume that $\mathcal{N} := \mathcal{K} \mathcal{L}$ is properly contained in $\mathcal{H}$. Then, since $\mathcal{K}$ is compact, the image
$\lambda(\mathcal{N}) = \lambda(\mathcal{K})$ is a  closed
subgroup of $\mathcal{M}$ that is properly contained in $\lambda(\mathcal{H})$. Since $\mathcal{M}$ is profinite, there exists an proper open subgroup $\mathscr{U} \subset \mathcal{M}$ that contains $\lambda(\mathcal{N})$ but not $\lambda(\mathcal{H})$. Then $\mathcal{V} := \lambda^{-1}(\mathscr{U})$ is a proper open
subgroup of $\mathcal{H}$ of finite index that contains $\mathcal{N}$. It follows that $\kappa(\mathcal{V})$ is a proper closed  subgroup of $\mathcal{G}$ of finite
index. Then the intersection of all conjugates of $\kappa(\mathcal{V})$ would be a proper \emph{normal} closed subgroup of $\mathcal{G}$ of finite index, which
by our assumption cannot exist. A contradiction.
\end{proof}

Applying the lemma to the homomorphisms $H_v \stackrel{\theta}{\longrightarrow} G(K_v)$ and $H_v \stackrel{\lambda_W}{\longrightarrow} \mathrm{Aut}(F/W)$, we obtain
that $\theta(\mathcal{L}_W) = G(K_v)$, for every open normal subgroup $W$ of $F$. Thus,
for any $g \in G(K_v)$, the fiber $\theta^{-1}(g)$ meets the closed subgroup $\mathcal{L}_W$.
Using the fact that $\mathcal{L}_{W_1} \cap \cdots \cap \mathcal{L}_{W_d} =
\mathcal{L}_{W_1 \cap \cdots \cap W_d}$ for any open normal subgroups $W_1, \ldots , W_d$ of $F$ and the compactness of $\theta^{-1}(g)$,
we conclude that
$$
\theta^{-1}(g) \bigcap \left( \bigcap_W \mathcal{L}_W \right) \neq \varnothing,
$$
where $W$ runs through all open normal subgroups of $F$. But $\bigcap_W \mathcal{L}_W$ clearly coincides with $\theta^{-1}(G(K_v)) \cap Z_H(F)$.
So, $g \in \theta(Z_H(F))$, implying that $G(K_v) \subset \theta(Z_H(F))$, hence $v \in \mathfrak{Z}(F).$
\end{proof}

\begin{prop}\label{P:v-in-Z2}
Let $V$ be a subset of $V^K \setminus S$, and let $T = V^K \setminus V$.
Assume that the congruence kernel $C^{(T \setminus \mathcal{A})}(G)$ is trivial and there exist subgroups $H_1$ and $H_2$ of $H$ such that

\medskip

\noindent \ $(i)$ \parbox[t]{15.5cm}{$\theta(H_1)$ and $\theta(H_2)$ are dense subgroups of $G(\mathbb{A}(S\cup V))$ and $G(\mathbb{A}(T))$,
respectively;}

\vskip1mm

\noindent $(ii)$ \parbox[t]{15.5cm}{$H_1$ and $H_2$ commute elementwise and together generate a dense subgroup of $H$.}

\medskip

\noindent Then $H_2$ centralizes $F$, and therefore $V \setminus \mathcal{A}$ is contained in $\mathfrak{Z}(F)$.

\end{prop}
\begin{proof}
We note that $G(\mathbb{A}(S)) = G(\mathbb{A}(T)) \times G(\mathbb{A}(S \cup V))$. Replacing the subgroups $H_1$ and $H_2$ with their closures, we may
assume that they are actually closed. Since $\theta$ is a closed map, we then see that
$$
\theta(H_1) = G(\mathbb{A}(S\cup V)) \ \ \text{and} \ \ \theta(H_2) = G(\mathbb{A}(T)).
$$
Now, we define the following closed normal subgroup of $H_1$:
$$
H_1'= H_1 \cap \theta^{-1}(G(\mathbb{A}(S \cup V \cup \mathcal{A}))).
$$
It follows from condition  $(ii)$ that the normalizer $N_{H}(H_1')$ contains $H_1$ and $H_2$, hence coincides with $H$. Thus, we can take the quotient
of the extension (\ref{E:F}) by $H_1'$, which yields the following exact sequence:
$$
1 \to F / (F \cap H_1') \longrightarrow H / H_1' \longrightarrow G(\mathbb{A}(T \setminus \mathcal{A})) \to 1.
$$
We now observe that this sequence inherits from (\ref{E:F}) a splitting over $G(K)$ whose image is dense in $H/H'_1$. Since the congruence kernel $C^{(T \setminus \mathcal{A})}(G)$ is trivial
by assumption, we conclude from Proposition \ref{P:Univ} that $F \subset H_1'$, which implies that $H_2$ centralizes $F$. Then for any $v \in V$ we have
$$
G(K_v) \subset \theta(H_2) \subset \theta(Z_H(F)),
$$
proving that $v \in \mathfrak{Z}(F)$ and establishing the inclusion $V \setminus \mathcal{A} \subset \mathfrak{Z}(F)$.
\end{proof}



\medskip

%
%
%
%
%
%
%
%
%
%
%
%

\medskip

Next, we will show how information about $\mathfrak{Z}$ can be used to conclude that (\ref{E:F}) is a central extension.

\begin{prop}\label{P:Centr1}
If there exists a subset $V$ of $\mathfrak{Z}(F)$
such that the congruence kernel $C^{(S \cup V)}(G)$ is trivial, then the extension {\rm (\ref{E:F})} is central. In particular,
{\rm (\ref{E:F})} is central whenever $\mathfrak{Z}(F) = V^K \setminus
(S \cup \mathcal{A})$.
\end{prop}
We begin with the following elementary lemma.

\begin{lemma}\label{L:A2}
Let
\begin{equation}\label{E:A-1}
1 \to \mathcal{F} \longrightarrow \mathcal{H}
\stackrel{\nu}{\longrightarrow}\mathcal{G} \to 1
\end{equation}
be an exact sequence of groups. Given a subgroup $\mathcal{H}'
\subset \mathcal{H}$ that centralizes $\mathcal{F}$ and an element
$a \in \mathcal{H}$ such that $\nu(a)$ centralizes
$\nu(\mathcal{H}'),$ the map $$\gamma_a \colon x \mapsto [a , x] =
axa^{-1}x^{-1} \ \ \text{for} \ \ x \in \mathcal{H}',$$ yields a group homomorphism $\gamma_a \colon \mathcal{H}' \to
\mathcal{F}.$
\end{lemma}
\begin{proof}
For $x \in \mathcal{H}'$ we have $\nu([a , x]) = [\nu(a) , \nu(x)] =
1$ implying that $\gamma_a(x) \in \mathcal{F}.$ Furthermore, for $x
, y \in \mathcal{H}'$ we have
$$
\gamma_a(xy) = [a , xy] = [a , x](x[a , y]x^{-1}) =
\gamma_a(x)\gamma_a(y),
$$
as required.
\end{proof}
\begin{cor}\label{C:A1}
Assume that {\rm (\ref{E:A-1})} is a central extension. Then for any
elementwise  commuting subgroups $\mathcal{G}_1 , \mathcal{G}_2
\subset \mathcal{G}$, the map $$c \colon \mathcal{G}_1 \times
\mathcal{G}_2 \to \mathcal{F}, \ \ (x , y) \mapsto
[\widetilde{x} , \widetilde{y}] \ \  \text{for} \ \  \widetilde{x} \in \nu^{-1}(x), \  \widetilde{y}
\in \nu^{-1}(y),$$ is a well-defined bimultiplicative pairing.
\end{cor}
Indeed, since $\mathcal{F}$ is central in $\mathcal{H}$, the commutator $[\widetilde{x} ,
\widetilde{y}]$ does not depend on the choice of lifts
$\widetilde{x} , \widetilde{y}$, making the map $c$  well-defined.  It follows
from the lemma that $c(\mathcal{G}_1 , \mathcal{G}_2)
\subset \mathcal{F},$ and that for any $x \in \mathcal{G}_1$ and
$y_1 , y_2 \in \mathcal{G}_2$ we have
$$
c(x , y_1y_2) = \gamma_{\widetilde{x}}(\widetilde{y}_1\widetilde{y}_2) =
\gamma_{\widetilde{x}}(\widetilde{y}_1)\gamma_{\widetilde{x}}(\widetilde{y}_2) =
c(x , y_1)c(x , y_2),
$$
proving that $c$ is multiplicative in the second variable. The multiplicativity in the first variable is established
by a similar computation.


\vskip3mm

{\it Proof of Proposition \ref{P:Centr1}.} To prove the first claim, we need to construct the subgroups
$H_1 , H_2$ of $H$ with the properties similar to those described in Proposition \ref{P:v-in-Z2}. Set
$H_1  = \theta^{-1}(G(\mathbb{A}(S \cup V))$ . To define $H_2$, we first consider $H' = \theta^{-1}(G(\mathbb{A}(T))) \cap Z_H(F)$, where $T =
V^K\setminus V$.  Clearly, the groups $G(K_v)$ for $v \in V$ generate
a dense subgroup of $G(\mathbb{A}(T))$. This fact has two implications relevant to our argument. First, since $V \subset \mathfrak{Z}(F)$ and $\theta$
is a closed map, the image $\theta(H')$ is a closed subgroup of $G(\mathbb{A}(T))$ containing $G(K_v)$ for all $v \in V$, hence $\theta(H') = G(\mathbb{A}(T))$.
Second, for any $v \in V$, the group $G$ is $K_v$-isotropic, and therefore $G(K_v)$ contains no proper normal subgroup of finite index (as we already mentioned above).
It follows that $G(\mathbb{A}(T))$ contains no proper \emph{closed} normal subgroup of finite index. Now, it follows from Lemma \ref{L:A2} that for any $a
\in H_1$, the map $\gamma_a \colon x \mapsto [a , x]$ defines a (continuous) group homomorphism $H' \to F$. We now consider the profinite group
$$
\mathcal{M} = \prod_{a \in H_1} F_a, \ \ \text{where} \ \ F_a = F \ \ \text{for \ all} \ \ a \in H_1,
$$
define a continuous homomorphism $\lambda \colon H' \to \mathcal{M}$, $x \mapsto (\gamma_a(x))$, and let $H_2 = \mathrm{Ker} \: \lambda$. Applying
Lemma \ref{L:A1}, we see that $\theta(H_2) = G(\mathbb{A}(T))$. This easily implies that $H = H_1 H_2$, and on the other hand, by our construction the
subgroups $H_1$ and $H_2$ commute elementwise. In particular, $H_2$ is normal in $H$, and hence (\ref{E:F}) gives rise
to the following exact sequence
$$
1 \to F/(F \cap H_2) \longrightarrow H/H_2 \longrightarrow G(\mathbb{A}(S))/G(\mathbb{A}(T)) = G(\mathbb{A}(S \cup V)) \to 1.
$$
Since $C^{(S \cup V)}(G)$ is trivial by assumption, Proposition \ref{P:Univ} implies that $F \subset H_2$. It follows that $H_1$ centralizes $F$,
and therefore so does $H = H_1H_2$ as $H_2 \subset Z_H(F)$ by our construction. (While this can be derived directly from Proposition \ref{P:v-in-Z2},
we gave an independent argument in order to avoid cumbersome notations.)

For the second assertion, we observe that for $V = V^K \setminus (S \cup \mathcal{A})$, the triviality of $C^{(S \cup V)}(G)$ is equivalent
to our standing assumption that (MP) holds for $G(K)$ and $S \cap \mathcal{A} = \varnothing$. \hfill $\Box$

\medskip

\begin{prop}\label{P:Com-lifts1}
Assume that $\mathcal{A} \cap S = \varnothing$ and  there is a partition
$V^K \setminus (S \cup \mathcal{A}) = \bigcup_{i \in I} V_i$ such that one can find subgroups $H_{\mathcal{A}}$ and
$H_i$ $(i \in I)$ of $H$ satisfying the following conditions:

\medskip

\noindent \ \ $(i)$ for each $i \in I$ and  $V'_i := \bigcup_{j \neq i} V_j$,  the congruence kernel  $C^{(S
\cup V'_i)}(G)$ is trivial;

\vskip1mm

\noindent \ $(ii)$ \parbox[t]{15.5cm}{$\theta(H_{\mathcal{A}})$ is a dense subgroup of $G_{\mathcal{A}} = \prod_{v \in \mathcal{A}} G(K_v)$
and $\theta(H_i)$ is a dense subgroup of $G(\A(V^K\setminus V_i))$ for all $i \in I$;}

\vskip1mm

\noindent $(iii)$ \parbox[t]{15.5cm}{any two of the subgroups $H_{\mathcal{A}}$ and $H_i$ for $i \in I$ commute elementwise;}

\vskip1mm

\noindent \ $(iv)$ \parbox[t]{15.5cm}{the subgroups $H_{\mathcal{A}}$ and $H_i$ for $i \in I$ generate a dense subgroup of $H$.}

\medskip

\noindent Then {\rm (\ref{E:F})} is a central extension.
\end{prop}
\begin{proof}
Fix $i \in I$, and apply Proposition \ref{P:v-in-Z2} with $V = \mathcal{A} \cup V_i$; then the corresponding $T$
is $S  \cup  V'_i$. We let $H(T)$ be the subgroup (of $H$) generated by $H_{\mathcal{A}}$ and $H_i$, and $H(S \cup T)$ be the subgroup  generated by $H_j$ for $j \in
I \setminus \{ i \}$. Using Proposition \ref{P:v-in-Z2}, we conclude that $V_i \subset \mathfrak{Z}(F)$. Since this is true for all $i \in I$, we see that actually
$\mathfrak{Z}(F) = V^K \setminus (S \cup \mathcal{A})$. Then (\ref{E:F}) is central by Proposition \ref{P:Centr1}. (We note that in the case $\mathcal{A} = \varnothing$,
the proof in fact does not require Proposition \ref{P:Centr1}.)
\end{proof}

\medskip

The following assertion goes back to \cite{Ra-Comb} (see also \cite{Lu1}, \cite[Proposition 7.1.3]{LuSe},  \cite{PR2}).
\begin{prop}\label{P:Prop-F}
If {\rm (\ref{E:F})} is not central then $F$ possesses closed
subgroups $F_2 \subset F_1,$ both of which are normal in $H,$ such
that the quotient $F/F_1$ is finite and the quotient $F_1/F_2$ is
isomorphic to $\prod_{i \in I} \Phi_i$ where $I$ is an {\rm
infinite} set and $\Phi_i = \Phi,$ the same finite simple group, for
all $i \in I.$ Consequently, if $F$ is finitely generated then it is
central, hence finite.
\end{prop}

\section{A criterion for centrality}

We begin with one additional notation. Let $v \in V^K$, fix a maximal $K_v$-torus $T$ of $G$, and let $T^{\mathrm{reg}}$
denote its Zariski-open subvariety of regular elements. It follows from the Implicit Function Theorem that the map
\begin{equation}\label{E:phi}
\varphi_{v , T} \colon G(K_v) \times T^{\mathrm{reg}}(K_v) \to G(K_v), \ \ (g , t) \mapsto gtg^{-1},
\end{equation}
is open; in particular, $$\mathscr{U}(v , T) := \varphi_{v , T}(G(K_v) \times T^{\mathrm{reg}}(K_v))$$ is an open
subset of $G(K_v)$. It follows from the definition that $\mathscr{U}(v , T)$ is \emph{conjugation-invariant} and \emph{solid}, i.e. intersects
every open subgroup of $G(K_v)$ (the latter property is primarily used when $v$ is nonarchimedean). Let $\theta$ be as in the short exact
sequence (\ref{E:F}) of the preceding section.

\begin{thm}\label{T:Criterion1}
$(i)$ If  extension $({\mathrm{F}})$ is central, then there exists a positive integer $n$ such that for any maximal
$K$-torus $T$ of $G$ and any $t \in T(K)$, we have the inclusion
\begin{equation}\label{E:Incl1}
\theta(Z_H(t)) \supset T(\mathbb{A}(S))^n
\end{equation}
$($here we view the group of $S$-adeles $T(\mathbb{A}(S))$ as a subgroup of $G(\mathbb{A}(S)) = \overline{G}$, and use $t$ to denote
also the lift of $t \in T(K)$ in $H$ provided by the splitting of {\rm (\ref{E:F})} over $G(K))$.

\medskip

\noindent $(ii)$ \parbox[t]{16.3cm}{Conversely, assume that there is an integer $n > 1$, a finite subset $V \subset V^K \setminus S$ and maximal
$K_v$-tori $T(v)$ of $G$ for $v \in V$ such that for any element  $t \in G(K) \cap U$ with $U = \prod_{v \in V} \mathscr{U}(v , T(v))$, which is regular semi-simple\footnotemark,    the inclusion (\ref{E:Incl1}) holds with $T = Z_G(t)^{\circ}$. Then {\rm(\ref{E:F})} is a central extension.}
\footnotetext{Of course, any $t \in G(K) \cap U$ is automatically regular semi-simple if $V \neq \varnothing$.}
%
%
\end{thm}

\medskip

\noindent {\it Proof of $(i)$.} Assume (\ref{E:F}) is central. Then the finiteness of the metaplectic kernel $M(S , G)$ \cite[Theorem 2.7]{PR1} implies
that $F$ is finite (cf.\,\cite[3.4--3.6]{PR-Milnor}). Set $n = \vert F \vert$. Now, let $T$ be a maximal $K$-torus of $G$, let $t \in T(K)$, and let $\mathscr{T} = \theta^{-1}(T(\mathbb{A}(S)))$.
By Lemma \ref{L:A2},  the map $\gamma_t \colon x \mapsto [t , x]$
yields a group homomorphism $\mathscr{T} \to F$.  It follows that $\gamma_t(\mathscr{T}^n) = \{ 1 \}$, i.e. $\mathscr{T}^n \subset Z_H(t)$. On the other
hand, $\theta(\mathscr{T}^n) = T(\mathbb{A}(S))^n$, and our assertion follows.  \hfill $\Box$

\medskip

For the proof of part $(ii)$ we need the following proposition in which we use $\overline{X}$ and $\widehat{X}$ to denote the closure of a subset $X \subset G(K)$
in $\overline{G}$ and $\widehat{G}$, respectively.
\begin{prop}\label{P:Central1}
{\rm (Cf.\,\cite[Proposition 3.2]{Rag-SL1})} Assume that there exist a positive integer $n$, a finite set of places $V \subset V^K \setminus S$ and maximal $K_v$-tori $T(v)$ of $G$ for
$v \in V$  such that for every regular semi-simple element $t \in G(K) \cap U$, where $U = \prod_{v \in V} \mathscr{U}(v , T(v))$, the inclusion (\ref{E:Incl1}) holds for $T= Z_G(t)^{\circ}$. Then for any normal subgroup $N$ of $\Gamma = G(\mathcal{O}(S))$
of finite index  and any $x \in \overline{N} \cap \Gamma$, we have
\begin{equation}\label{E:Incl2}
Z(N , x)(\overline{N} \cap \Gamma) \supset (\Gamma \cap U)^n,
\end{equation}
where $Z(N , x) := \{ \gamma \in \Gamma \, \vert \, [x , \gamma] \in N \}$.
\end{prop}
(Note that $Z(N , x)$ is simply the pullback of the centralizer of $xN$ in $\Gamma/N$ under the canonical homomorphism $\Gamma \to \Gamma/N$.)
\begin{proof} For proving (\ref{E:Incl2}), we may replace $N$ with a smaller normal subgroup of $\Gamma$ of finite index to assume that $\overline{N} =
\prod_{v \notin S} N_v$, where $N_v \subset G(\mathcal{O}_v)$ is an open normal subgroup for all $v \notin S$, and $N_v = G(\mathcal{O}_v)$ for all
$v \in V^K \setminus (S \cup V')$ for a suitable finite \emph{nonempty} subset $V' \subset V^K \setminus S$.

We need to show that for any $z \in \Gamma \cap U$, we have
\begin{equation}\label{E:Incl7}
z^n \in Z(N , x)(\overline{N} \cap \Gamma).
\end{equation}
If $V \neq \varnothing$, then $z$ is automatically regular semi-simple. If $V = \varnothing$, and hence $\Gamma \cap U = \Gamma$, using the Zariski-density
of $N$ in $G$ in conjunction with the fact that the set of regular semi-simple elements is Zariski-open in $G$, we see that the coset $zN$ contains a regular
semi-simple element $z' \in \Gamma \cap U$. This means that in proving (\ref{E:Incl7}), we may assume $z$ to be regular semi-simple.
Let $T_0 = Z_G(z)^{\circ}$ be the maximal $K$-torus of $G$ containing $z$.
For $v \in V'' := V \cup V'$ consider the open set $W_v := \varphi_{v , T_0}(N_v \times T^{\mathrm{reg}}_0(K_v))$ of $G(K_v)$ (see (\ref{E:phi})), and then set
$$
\mathcal{W} = \prod_{v \in V''} W_v \ \ \ \text{and} \ \ \ \mathcal{N} = \prod_{v \in V''} N_v.
$$
Since $N_v$ is an open subgroup, the intersection $N_v \cap T_0(K_v)$ meets $T^{\mathrm{reg}}_0(K_v)$ for every $v \in V''$, which implies that $\mathcal{W} \cap \mathcal{N}$
is a \emph{nonempty} open subset of $\mathcal{N}$, and hence $\mathcal{N} \subset \mathcal{W}N$. Since $x \in \overline{N} \cap \Gamma$, there exist $y \in N$ such that
\begin{equation}\label{E:Incl15}
xy = gbg^{-1}
\end{equation}
for some $g = (g_v)$ with $g_v \in N_v$ and $b = (b_v)$ with $b_v \in T^{\mathrm{reg}}_0(K_v)$, for $v \in V''$. As $V'' \neq \varnothing$, the element $t := xy$ is automatically
regular semi-simple. Let $T = Z_G(t)^{\circ}$ be the maximal $K$-torus of $G$ containing $t$.

For $v \notin S$, define
$$
a_v = \left\{\begin{array}{ccl} g_vz^ng_v^{-1} & \text{if} & v \in V'', \\ 1 & \text{if} & v \notin S \cup V''. \end{array}       \right.
$$
It follows from (\ref{E:Incl15}) that the adele $a = (a_v)$ belongs to $T(\mathbb{A}(S))^n$. So, by (\ref{E:Incl1}) there exists $s \in Z_H(t)$ such that
$\theta(s) = a$. In fact, $a \in \overline{\Gamma}$, so $s \in \widehat{\Gamma}$, and $\Gamma\cap s{\widehat{N}}$ is nonempty; we pick $c \in \Gamma \cap s\widehat{N}$. Then
$$
[x , c] \in [t , c]\widehat{N} = [t , s]\widehat{N} = \widehat{N}
$$
implying that $c \in Z(N , x)$ (note that, being of finite index in $\Gamma$, the normal subgroup $N$ is open (and hence closed) in the profinite topology on the former, and hence $\Gamma\cap {\widehat{N}} = N$). On the other hand,
$$
c \in \theta(s)\overline{N} = a\overline{N} = z^n\overline{N}
$$
as $az^{-n} \in \overline{N}$ because $a_vz^{-n} = [g_v , z^n] \in N_v$ since $g_v \in N_v$, for $v \in V''$, and $a_vz^{-n} = z^{-n} \in G(\mathcal{O}_v) = N_v$,
for $v \in V^K \setminus (S \cup V'')$. \end{proof}

\medskip

{\it Proof of $(ii)$ in Theorem \ref{T:Criterion1}.} First, we will derive from Proposition \ref{P:Central1} that for any $x \in F$ we have the inclusion
\begin{equation}\label{E:Incl3}
\theta(Z_H(x)) \supset (\overline{\Gamma \cap U})^n.
\end{equation}
Consider the profinite group $\Delta = \theta^{-1}(\overline{\Gamma})$, which is  a quotient of $\widehat{\Gamma}$, and take any $\gamma \in \Gamma \cap U$ (we will canonically
identify $\Gamma$ with a dense subgroup of $\Delta$ using the splitting of $\theta$ over $G(K)$).
Let $\mathscr{R}$ be the family of all open normal subgroups of $\Delta$. For $R \in \mathscr{R}$,  set $$N_R := \Gamma \cap R \ \ \  \text{and} \ \ \ \widetilde{R} := \theta^{-1}(\overline{N_R}),$$
and pick $x_R \in \Gamma \cap (xR)$. Applying Proposition \ref{P:Central1} to $N_R$ and $x_R$, we obtain that
\begin{equation}\label{E:Incl4}
\gamma^n\widetilde{R} \bigcap \widetilde{Z}(R , x) \neq \varnothing \ \ \ \text{for any} \ \ R \in \mathscr{R},
\end{equation}
where $\widetilde{Z}(R , x) := \{ \delta \in \Delta \vert [x , \delta] \in R \} = \{ \delta \in \Delta \vert [x_R , \delta] \in R \}$. Using the compactness of $F$,
one easily derives from this that
\begin{equation}\label{E:Incl5}
\gamma^n F \bigcap Z_{\Delta}(x) \neq \varnothing.
\end{equation}
Indeed, one observes that
$$
\bigcap_{R \in \mathscr{R}} \widetilde{R} = F, \ \ \ \bigcap_{R \in \mathscr{R}} \widetilde{Z}(R , x) = Z_{\Delta}(x),
$$
and for any $R_1, \ldots , R_d \in \mathscr{R}$, we have $$\widetilde{Z}(R_1 , x) \cap \cdots \cap \widetilde{Z}(R_d , x) = \widetilde{Z}(R_1 \cap \cdots \cap R_d , x).$$ So, if (\ref{E:Incl5}) does not hold,
there exists $R' \in \mathscr{R}$ such $\gamma^n F \cap Z(R' , x) = \varnothing$. Next, using the compactness of $Z(R' , x)$, we see that there exists $R''  \in \mathscr{R}$
such that $\gamma^n R'' \cap Z(R' , x) = \varnothing$. Then for $R = R' \cap R''$, the inclusion (\ref{E:Incl4}) fails to hold, a contradiction.

We have proved that $(\Gamma\cap U)^n \subset \theta(Z_{\Delta}(x))$. Since $\theta(Z_{\Delta}(x))$ is closed, passing to the closure, we obtain (\ref{E:Incl3}).
%
%
%
Furthermore, we have  ${\overline{\Gamma\cap U}} = \prod_{v \notin S} {\overline{\Omega_v}}$, where $\Omega_v = G(\mathcal{O}_v) \cap \mathscr{U}(v,T(v))$ for $v \in V$ and $\Omega_v =
G(\mathcal{O}_v)$ for $v \in V^K \setminus (S \cup V)$. In all cases, $\Omega_v$ is a \emph{nonempty} open subset of $G(K_v)$, hence Zariski-dense. It follows that
$(\Omega_v)^n$ is always infinite. Now,  we conclude from (\ref{E:Incl3}) and
Proposition \ref{P:v-in-Z1} that $\mathfrak{Z}(F)$ equals $V^K \setminus (S \cup \mathcal{A})$. Then the extension (\ref{E:F}) is central by Proposition~\ref{P:Centr1}. \hfill $\Box$

\section{First applications and proof of Theorem A}\label{S:Appl}

To verify the inclusion (\ref{E:Incl1}) in Theorem \ref{T:Criterion1}, we observe that for $t \in T(K)$, the centralizer
$Z_H(t)$ contains $T(K)$, hence the closure $\widehat{T(K)}$, and therefore $\theta(Z_H(t))$ contains $\theta(\widehat{T(K)}) = \overline{T(K)}$.
So, we could immediately derive the centrality of (\ref{E:F}) using Theorem \ref{T:Criterion1} if we knew that there exists
an integer $n > 0$ such that for any maximal $K$-torus $T$ of $G$ (or at least for any maximal $K$-torus with specified local behavior at finitely many
places), the quotient $T(\mathbb{A}(S))/\overline{T(K)}$ has exponent dividing $n$ (``almost strong approximation property'' up to exponent $n$). Unfortunately, when
$S$ is finite the latter quotient has infinite exponent (cf.\,\cite[Proposition 4]{PR-Irred}), which forces us to use some additional considerations
(cf. Proposition \ref{P:Com-lifts2} and Examples 4.6 and 4.7 below). In the next section, we will establish almost strong approximation property in the case where $S$ contains all but finitely many elements of a generalized arithmetic progression  (see
Theorem \ref{T:SA-tori}), which will lead to Theorem B of the introduction. In this section we will consider separately a basic case where $V := V^K \setminus S$ is finite (i.e., $S$
is cofinite) as this case  has some interesting consequences (like Theorem A of the introduction). Since in this case the corresponding ring of $S$-integers $\mathcal{O}(S)$
is the intersection of finitely many discrete valuation subrings of $K$ corresponding to the places in $V$, hence is semi-local, we will refer to this case as \emph{semi-local}.

\medskip

We begin with the following proposition which was already implicitly established in \cite[\S 9]{PR1}.
\begin{prop}\label{P:WA-tori} {\rm (Almost weak approximation)}
For every $d \geqslant 1$, there exists an integer $n = n(d) \geqslant 1$ such that given a $K$-torus $T$ of dimension $\leqslant d$, for any finite
set of places $V \subset V^K$, the quotient $T_V / \overline{T(K)}$, where $T_V = \prod_{v \in V} T(K)$ and $\overline{T(K)}$ denote the closure of
$T(K)$ in $T_V$, has exponent dividing $n$.
\end{prop}
\begin{proof}
Pick $n = n(d)$ so that it is divisible by the order of any finite subgroup of the group $\mathrm{GL}_d(\Z)$ (it follows from Minkowski's lemma that one can take
$n$ to be the index in $\mathrm{GL}_d(\Z)$ of the principal congruence subgroup modulo $3$). Let $T$ be an arbitrary $K$-torus of dimension $m \leqslant d$. We let $E := K_T$
denote the minimal splitting field of $T$ over $K$, and set $\mathcal{G} = \mathrm{Gal}(E/K)$. The natural action of $\mathcal{G}$ on the character group $X(T)$ defines its faithful representation in $\mathrm{GL}_m(\Z)$, so the order $\vert \mathcal{G} \vert$ divides $n(d)$. Then, for the dual module of co-characters $X_*(T)$, we find a surjective
homomorphism $\phi \colon \Z[\mathcal{G}]^{\ell} \to X_*(T)$, and let $M = \mathrm{Ker}\: \phi$. Let $T'$ and $T''$ be the $K$-tori that split over $E$ and have $\Z[\mathcal{G}]^{\ell}$ and $M$ as their co-character modules; clearly, $T' = \mathrm{R}_{E/K}(\mathrm{GL}_1)^{\ell}$, hence it is quasi-split.  We have the following exact sequence of $K$-tori: $$
1 \to T'' \longrightarrow T' \stackrel{\eta}{\longrightarrow} T \to 1.
$$
This sequence gives rise to the following commutative diagram with exact bottom row:
$$
\begin{array}{ccccc}
T'(K) & \stackrel{\eta_K}{\longrightarrow} & T(K) & &   \\
\downarrow & & \downarrow & &  \\
T'_V & \stackrel{\eta_V}{\longrightarrow} & T_V & \longrightarrow & \prod_{v \in V} H^1(K_v , T'')
\end{array}.
$$
Being quasi-split, hence rational over $K$, the torus $T'$ has weak approximation property with respect to any finite set of places (cf.\,\cite[Proposition 7.3]{PlRa}), i.e. $\overline{T'(K)} =
T'_V$. It follows that $\overline{T(K)}$ contains $\eta_V(T'_V)$. On the other hand, the quotient $T_V /\eta_V(T'_V)$ embeds into $\prod_{v \in V} H^1(K_v , T'')$. But for $v \in V$, as a consequence of Hilbert's Theorem 90 for tori we have $H^1(K_v , T'') = H^1(\mathcal{G}_w , T''(L_w))$ where $\mathcal{G}_w$ is the decomposition group $\mathrm{Gal}(E_w/K_v)$ for some extension $w \vert v$. By our construction, the order $\vert \mathcal{G}_w \vert$ divides $n$. Therefore, the quotient $T_V / \eta_V(T'_V)$ has exponent dividing $n$, and our claim follows.  (We note that the proof enables us to somewhat optimize our choice of $n$: all we need is that $n$ be divisible by the order of any finite \emph{solvable} subgroup
of $\mathrm{GL}_d(\Z)$.)
\end{proof}

\vskip2mm

\begin{cor}\label{C:WA-tori}
Let $G$ be a reductive $K$-group. There exists $n \geqslant 1$ such that for any maximal $K$-torus $T$ of $G$ and any finite set of places $V \subset V^K$,
the quotient $T_V/\overline{T(K)}$ has exponent dividing $n$.
\end{cor}

\vskip2mm

Now, let $G$ be an absolutely almost simple simply connected algebraic group over a global field $K$, and let $V$ be a finite set of nonarchimedean places of $K$ containing 
the set $\mathcal{A}$ of anisotropic places. Set $S = V^K \setminus V$. Then for any maximal $K$-torus $T$ of $G$ the group $T(\mathbb{A}(S))$ can be identified with the group 
$T_V$ in the above notations. Thus, Corollary \ref{C:WA-tori} asserts the existence of $n \geqslant 1$ (independent of $T$) such that the closure $\overline{T(K)}$ of $T(K)$ in 
$T(\mathbb{A}(S))$ contains $T(\mathbb{A}(S))^n$, for \emph{any} maximal $K$-torus $T$ of $G$. Let us use this fact to analyze  the congruence sequence (\ref{E:C}) appearing in \S 2. 
As we observed at the beginning of this section, for a maximal $K$-torus $T$ of $G$ and any $t \in T(K)$, the image $\pi(Z_{\widehat{G}}(t))$ of the corresponding centralizer 
contains $\overline{T(K)}$, hence $T(\mathbb{A}(S))^n$. This enables us to use Theorem \ref{T:Criterion1} to conclude that the congruence kernel  $C^{(S)}(G)$ is central.  Furthermore, it follows from our computations of the metaplectic kernel \cite{PR1} that in the situation at hand $M(S , G)$ is trivial, so being central $C^{(S)}(G)$ is actually trivial (provided that (MP) holds for $G(K)$, which we assume). Thus, we obtain the following:

\begin{thm}\label{T:SL}
Let $G$ be an absolutely almost simple simply connected algebraic group over a global field $K$, and assume that $({\mathrm{MP}})$ holds for $G(K)$. Then for
any finite set $V$ of nonarchimedean places of $K$ that contains the set $\mathcal{A}$ of anisotropic places and $S = V^K \setminus V$, the congruence kernel $C^{(S)}(G)$ is central and hence trivial.
\end{thm}

\vskip2mm

\begin{remark}
Sury \cite{Sury} showed that for the absolutely almost simple simply connected anisotropic groups of type $\textsf{A}_1$ as well as simply connected groups of classical types
associated with bilinear and certain hermitian/skew-hermitian forms, the methods used to prove (MP) (see \cite[Chapter 9]{PlRa}) can be
adapted to prove Theorem \ref{T:SL}. This does not appear to be the case for the anisotropic inner forms of type $\textsf{A}_n$ with $n > 1$, i.e. for
the groups of the form $G = \mathrm{SL}_{1 , D}$, where $D$ is a central division algebra over $K$ of degree $d > 2$.  Indeed, in this case the proof of (MP)
is derived from the following result which is valid over \emph{any} field: {\it Let $D$ be a finite-dimensional division algebra over a field $K$. Then $D^{\times}$
cannot have a nonabelian finite simple group as a quotient} (see \cite{Segev}, and also \cite{RS}). {\it In fact, every finite quotient of $D^{\times}$ is solvable} \cite{RSS}.
(See also \cite{R-MP} for another proof of (MP) along these lines.) All these results rely on the following fact: {\it For a finite index subgroup $N$ of $D^{\times}$, we have
$D = N - N$} (\cite{Ber-Shap},\cite{Tur}). However, there is no valid analog of this fact for finite-index subgroup of $\mathcal{D}^{\times}$, where $\mathcal{D}$ is an order in $D$ over a semi-local subring $\mathcal{O}$ of $K$ that has finite homomorphic images (see \cite{Ber-Shap} regarding the case where $\mathcal{D}$ has no such images).
\end{remark}

\vskip2mm

Combining Theorem \ref{T:SL} with Proposition \ref{P:Com-lifts1}, we obtain the following.
\begin{prop}\label{P:Com-lifts2}
Assume that $\mathcal{A} \cap S = \varnothing$ and  there is a partition
$V^K \setminus (S \cup \mathcal{A}) = \bigcup_{i \in I} V_i$, with all $V_i$'s \emph{finite}, such that one can find subgroups $H_{\mathcal{A}}$ and
$H_i$ $(i \in I)$ of $H$ satisfying the following conditions:

\vskip2mm

%
%

\noindent \ \ $(i)$ \parbox[t]{15.5cm}{$\theta(H_{\mathcal{A}})$ is a dense subgroup of $G_{\mathcal{A}} = \prod_{v \in \mathcal{A}} G(K_v)$
and $\theta(H_i)$ is a dense subgroup of $G(\A(V^K\setminus V_i))$ for all $i \in I$;}

\vskip1mm

\noindent \ $(ii)$ \parbox[t]{15.5cm}{any two of the subgroups $H_{\mathcal{A}}$ and $H_i$ for $i \in I$ commute elementwise;}

\vskip1mm

\noindent $(iii)$ \parbox[t]{15.5cm}{the subgroups $H_{\mathcal{A}}$ and $H_i$ for $i \in I$ generate a dense subgroup of $H$.}

\medskip

\noindent Then {\rm (\ref{E:F})} is a central extension.
\end{prop}

\medskip

\noindent {\it Proof of Theorem A}. We apply Proposition \ref{P:Com-lifts2} to $H = \widehat{G}^{(S)}$ and $F = C^{(S)}(G)$  by considering the partition of $V^K \setminus (S \cup \mathcal{A})$ into one-element subsets (singletons). We let $H_{\mathcal{A}}$ be the subgroup generated by $\mathcal{G}_v$ (notations as in the statement of Theorem A) for $v\in \mathcal{A}$, and set $H_v = \mathcal{G}_v$ for $v \in V^K \setminus (S \cup \mathcal{A})$. Then the assumptions of Theorem A immediately show that the conditions of Proposition \ref{P:Com-lifts2} are satisfied and the centrality of $C^{(S)}(G)$ follows. \hfill $\Box$

\medskip

We will now show how Theorem A can be used to establish the centrality of $C^{(S)}(G)$ in some known cases.

\medskip

\addtocounter{thm}{2}

\noindent {\bf Example 4.6.}  Let $G = \mathrm{SL}_{n}$ with $n \geqslant 3$, and $S \subset V^K$ be an arbitrary subset containing $V^K_{\infty}$.
The first proof of centrality in this case was given by Bass, Milnor and Serre in \cite{BMS}.
In order to apply Theorem A and give an alternative argument, for $1 \leqslant i , j \leqslant
n$, $i \neq j$, we consider the corresponding 1-dimensional unipotent subgroup $U_{ij}$ of $G$ together with its canonical parametrization
$e_{ij} \colon \mathbb{G}_a \to U_{ij}$. The following commutation relation
for elementary matrices is well-known:
\begin{equation}\label{E:Elem}
[e_{ij}(s) , e_{lm}(t)] = \left\{ \begin{array}{ll} 1, & i \neq m, j
\neq l \\ e_{im}(st), & j = l, i \neq m \\ e_{lj}(-st), & j \neq l,
i = m \end{array} \right.
\end{equation}
It is easy to see that the topologies $\tau_a$ and $\tau_c$ of
$G(K)$ induce the same topology on each $U_{ij}(K)$ (cf.\:Theorem
7.5(e) in \cite{BMS}). So, if $\widehat{U}_{ij}$ and
$\overline{U}_{ij}$ denote the closures of $U_{ij}(K)$ in
$\widehat{G}$ and $\overline{G},$ respectively, then $\widehat{G}
\stackrel{\pi}{\longrightarrow} \overline{G}$ restricts to an
isomorphism $\widehat{U}_{ij} \stackrel{\pi_{ij}}{\longrightarrow}
\overline{U}_{ij}.$ By the strong approximation property for the additive group $\mathbb{G}_a$,
the isomorphism $(e_{ij})_K \colon K^+ \to U_{ij}(K)$ extends to an isomorphism $\overline{e}_{ij} \colon \mathbb{A}(S) \to
\overline{U}_{ij}$. Then
$\widehat{e}_{ij} := \pi_{ij}^{-1} \circ \overline{e}_{ij}$ is an
isomorphism  $\mathbb{A}(S) \to\widehat{U}_{ij}$. We will
let $\cG_v$, for $v \notin S$, denote the subgroup of $\widehat{G}$
generated by $\widehat{e}_{ij}(t)$ for all $t\in K_v \subset
\mathbb{A}_S$ and all $i \neq j.$ Clearly, the $\cG_v$'s satisfy
condition $(i)$ of Theorem A. Since $K_v$ for $v \in V^K \setminus S$ additively
generate a dense subgroup of $\mathbb{A}(S)$, the closed subgroup of $\widehat{G}$
generated by the $\cG_v$, $v\notin S$, contains
$\widehat{e}_{ij}(\mathbb{A}(S))$ for all $i \neq j$. In particular, it contains
$\widehat{e}_{ij}(K)$ for all $i \neq j$, hence $G(K)$, and therefore coincides
with $\widehat{G}$, verifying condition $(iii)$. Finally, to check $(ii)$, we observe that the density of $K$ in
$\mathbb{A}(S)$ implies that (\ref{E:Elem}) entails a similar expression
for $[\widehat{e}_{ij}(s) , \widehat{e}_{lm}(t)]$ for any $s , t \in
\mathbb{A}(S)$. Now, for $s \in K_{v_1}$ and $t \in K_{v_2}$, where $v_1
\neq v_2$, we have $st = 0$ in $\mathbb{A}(S)$, which implies that
$\widehat{e}_{ij}(s)$ and $\widehat{e}_{lm}(t)$ commute except
possibly when $l = j$ and $m = i.$ In the latter case, as $n
\geqslant 3,$ we can pick $l \neq i , j$ and then write
$\widehat{e}_{ji}(t) = [\widehat{e}_{jl}(t) ,
\widehat{e}_{li}(1_{K_{v_2}})].$ Since $\widehat{e}_{ij}(s)$ is
already known to commute with $\widehat{e}_{jl}(t)$ and
$\widehat{e}_{li}(1_{K_{v_2}}),$ it commutes with
$\widehat{e}_{ji}(t)$ as well. This shows that $\cG_{v_1}$ and
$\cG_{v_2}$ commute elementwise, which verifies condition $(ii)$ of Theorem A.
Then the latter yields the centrality of $C^{(S)}(G)$.

(We note that the idea of using commuting lifts of ``local'' groups is useful
in the analysis of the congruence subgroup problem not only in the context of
algebraic groups over the rings of $S$-integers in global fields,
it was used in \cite{RR} together  with the result of M.\:Stein \cite{Stein} on
the centrality of $K_2$ over semi-local rings to prove the centrality of the congruence
kernel for elementary subgroups of Chevalley groups of rank $> 1$ over arbitrary Noetherian
rings. It is worth noting that the above argument based on almost weak approximation in maximal tori
and the action of the group of rational points on the congruence kernel enables one to bypass
the rather technical computations of Stein, but the exact trade-off between these two approaches
is not apparent.) \hfill $\Box$

\medskip

\noindent {\bf Example 4.7.} Let $G = \mathrm{SL}_2$, and let $S \subset V^K$ be a subset that contains $V^K_{\infty}$
and is of size $\vert S \vert > 1$; by Dirichlet's Unit Theorem (see \cite[Ch. 2, Theorem 18.1]{ANT}), the latter is equivalent to the existence
of a unit $\varepsilon \in \mathcal{O}(S)^{\times}$ of infinite order.  The centrality of $C^{(S)}(G)$ in this case was first established by Serre \cite{Serre1}.
We will now show that this can also be derived from Theorem A. (We note that the argument below, unlike Serre's original proof, makes no use
of Tchebotarev's Density Theorem.) We let $U^+$, $U^-$ and $T$ denote the subgroups of upper and lower unitriangular matrices and of diagonal matrices,
respectively, and fix the following standard parametrizations of these groups:
$$
u^+(a) = \left(\begin{array}{cc} 1 & a \\ 0 & 1 \end{array}\right)\
, \ u^-(b) = \left(\begin{array}{cc} 1 & 0 \\ b & 1
\end{array}\right)\ , \ h(t) = \left(\begin{array}{cc}
t & 0 \\ 0 & t^{-1} \end{array}\right),
$$
($a , b \in \mathbb{G}_a$, $t \in \mathrm{GL}_1$). For $a , b \in K$ such that $ab \neq 1$,
one easily verifies the following commutator identity:
\begin{equation}\label{E:Elem2}
[u^+(a) , u^-(b)] = u^+\left(-\frac{a^2b}{1 - ab}\right) h\left(
\frac{1}{1 - ab} \right) u^-\left(\frac{ab^2}{1 - ab} \right).
\end{equation}
We let $\widehat{U}^{\pm}$ and $\overline{U}^{\pm}$ denote
the closures of $U^{\pm}(K)$ in $\widehat{G}$ and $\overline{G},$
respectively. Again, it is easy to check that the topologies $\tau_a$ and $\tau_c$
of $G(K)$ induce the same topology on $U^+(K)$ and $U^-(K)$ (cf.\,\cite[1.4, Prop.\,1]{Serre1}),
so $\widehat{G} \stackrel{\pi}{\longrightarrow} \overline{G}$ restricts to
isomorphisms $\widehat{U}^{\pm} \stackrel{\pi^{\pm}}{\longrightarrow} \overline{U}^{\pm}$.
Furthermore, $(u^{\pm})_K$ extend to isomorphisms $\overline{u}^{\pm} \colon \mathbb{A}(S) \to
\overline{U}^{\pm}$. So, the maps $\widehat{u}^{\pm} := (\pi^{\pm})^{-1} \circ \overline{u}^{\pm}$ give
isomorphisms $\mathbb{A}(S) \to \widehat{U}^{\pm}$. For $v \in V^K \setminus S$, we  let
$\cG_{v}$ denote the subgroup of $\widehat{G}$ generated by
$\widehat{u}^+(K_v)$ and $\widehat{u}^-(K_v)$. As in Example 4.6, one checks
that the subgroups $\cG_v$ clearly satisfy conditions $(i)$ and $(iii)$ of Theorem A,
so we only need to verify condition $(ii)$. In other words, we need
to show that for $v_1 \neq v_2,$ the subgroups
$\widehat{u}^+(K_{v_1})$ and $\widehat{u}^-(K_{v_2})$ commute
elementwise.

First, we construct {\it nonzero} $a_0 \in K_{v_1}$ and
$b_0 \in K_{v_2}$ such that $\widehat{u}^+(a_0)$ and
$\widehat{u}^-(b_0)$ commute in $\widehat{G}.$ Let us enumerate the valuations
in $V^K \setminus (S \cup \{v_1 , v_2\})$ as $v_3, v_4, \ldots$. If $d$ is the class number of
$\mathcal{O}(S)$, then for each $i = 1, 2, 3, \ldots$, we can pick an element $p_i \in \mathcal{O}(S)$
such that $v_i(p_i) = d$ and $v_j(p_i) = 0$ for $j \neq i$.  Fix a unit $\varepsilon \in \mathcal{O}(S)^{\times}$ of infinite
order. Then for any $m \geqslant 2$ we can find
an integer $n(m)$ divisible by $m!$ so that
$$
\varepsilon^{n(m)} \equiv 1\:(\md (p_1 \cdots p_m)^{2m}).
$$
We can then write $1 - \varepsilon^{n(m)} = a_mb_m$ with $a_m , b_m \in \mathcal{O}(S)$ satisfying
\begin{equation}\tag{a}\label{E:a}
a_m \equiv 0\:(\md (p_2 \ldots p_m)^m), \ v_1(a_m) < d,
\end{equation}
and
\begin{equation}\tag{b}\label{E:b}
b_m \equiv 0\: (\md (p_1p_3 \cdots p_m)^m), \ v_2(b_m) < d.
\end{equation}
Since $a_m , b_m \in \mathcal{O}(S)$, there exists a subsequence $\{ m_j \}$ such that $a_{m_j} \to a_0$ and $b_{m_j} \to b_0$
in $\mathbb{A}(S)$. In fact, it follows from (\ref{E:a}) and (\ref{E:b}) that $a_0 \in K_{v_1}^{\times}$ and $b_0 \in K_{v_2}^{\times}$.
To show that $\widehat{u}^+(a_0)$ and $\widehat{u}^-(b_0)$ commute, we observe that
$$
[\widehat{u}^+(a_0) , \widehat{u}^-(b_0)] = \lim_{j \to \infty} [u^+(a_{m_j}) , u^-(b_{m_j})] \ \ \text{in} \ \ \widehat{G}.
$$
On the other hand, using (\ref{E:Elem2}), we obtain
$$
[u^+(a_m) , u^-(b_m)] =
u^+(-a_m^2b_m\varepsilon^{-n(m)})h(\varepsilon^{-n(m)})u^-(a_mb_m^2\varepsilon^{-n(m)})
\rightarrow 1 \ \text{in} \ \widehat{G},
$$
because $a_m^2b_m ,\: a_mb_m^2 \longrightarrow 0$ in $\mathbb{A}(S)$,
and $h(\varepsilon^{n(m)}) \longrightarrow 1$ in $\widehat{G}$ as $n(m)$ is
divisible by $m!$ and hence $h(\varepsilon^{n(m)})$ belongs to any given finite
index normal subgroup $N$ of $G(\mathcal{O}(S))$ for all sufficiently large $m$.
Thus, $[\widehat{u}^+(a_0) , \widehat{u}^-(b_0)]
= 1$. Now, for $t \in K^{\times},$ the automorphism $\sigma_t$ of
$G$ given by conjugation by $\mathrm{diag}(t , 1)$ extends to an
automorphism $\widehat{\sigma}_t$ of $\widehat{G}$. Then
$$
1 = \sigma_t([\widehat{u}^+(a_0) , \widehat{u}^-(b_0)]) =
[\widehat{u}^+(ta_0) , \widehat{u}^-(t^{-1}b_0)]
$$
for any $t \in K^{\times}$.  Since $K^{\times}$ is dense in
$K_{v_1}^{\times} \times K_{v_2}^{\times}$ by weak approximation, we
obtain that $[\widehat{u}^+(a) , \widehat{u}^-(b)] = 1$ for {\it
all}  $a \in K_{v_1},$ $b \in K_{v_2},$ as required. \hfill $\Box$

\medskip

\begin{remark}
The argument given in Example 4.6 can be generalized to prove the centrality of $C^{(S)}(G)$ for any absolutely almost simple simply connected algebraic
$K$-group $G$ with $\mathrm{rk}_K \: G \geqslant 2$. The first proof of this fact was given by M.S.\:Raghunathan in \cite{Ra1}; a shorter argument was given in \cite{Ra3}.
The case where $\mathrm{rk}_K \: G = 1$ and $\mathrm{rk}_S \: G \geqslant 2$ (which generalizes Example 4.7) is more complicated;
it was treated by Raghunathan in \cite{Ra2} by a different method.
One can give an alternative (shorter) argument (at least when $\mathrm{char}\: K \neq 2$) based on Proposition \ref{P:Com-lifts2}; details will be published elsewhere.
Theorem A can also be used to simplify the proof of Serre's conjecture for some anisotropic exceptional groups \cite{Ra-CSP}.
\end{remark}

\section{Strong approximation property in tori with respect to arithmetic progressions and the proof of Theorem B}

Strong approximation property in tori with respect to (generalized) arithmetic progressions was analyzed in \cite{PR-Irred}, and we begin by reviewing
some of the results obtained therein (we refer the reader to \cite{Dem} and references therein for the analysis of strong approximation from a different
perspective). Let $\cP(F/K , \mathscr{C})$  be a generalized arithmetic progression, where $F/K$ is a finite Galois extension with
Galois group $\mathscr{G}$, and $\mathscr{C}$ is a conjugacy class in $\mathscr{G}$ (for definition see \S\ref{S:Intro}). For a finite extension $E/K$, we let $\mathbb{I}_E$
denote the group of ideles of $E$. Furthermore,  given a subset $S$ of $V^K$, we let $\overline{S}$ denote the set of all extensions of places from $S$ to $E$, and then
let $\mathbb{I}_E(\overline{S})$ denote the group of $\overline{S}$-ideles and let $\overline{E^{\times}}^{(\overline{S})}$ be the closure of (the diagonally embedded) $E^{\times}$
in $\mathbb{I}_E(\overline{S})$.
%
%
%
%
\begin{prop}\label{P:GAP1}
{\rm (Cf.\,\cite[Proposition 3]{PR-Irred})} Let $\cP(F/K , \mathscr{C})$ be a generalized arithmetic
progression, $\cP_0 \subset \cP(F/K ,
\mathscr{C})$ be a finite (possibly, empty) subset, and let
$$
S = (\cP(F/K , \mathscr{C}) \setminus
\cP_0) \cup V^K_{\infty}.
$$
Furthermore, let $E/K$ be a finite separable extension. If $\mathscr{C}$ contains an automorphism that acts trivially on $E \cap F$
(in particular, if $\mathscr{C} = \{ e \}$ or $E \cap F =
K$) then the index
$$
\left[ \mathbb{I}_E(\overline{S}) :
\overline{E^{\times}}^{(\overline{S})} \right] \ \ \text{is finite
and divides} \ \ [F : K].
$$
\end{prop}
\begin{proof}
By the reduction theory for ideles (cf.\,\cite[Ch.\,2, \S 16]{ANT}), the quotient $\mathbb{I}^1_E/E^{\times}$, where $\mathbb{I}^1_E$ is the group
of ideles with content $1$, is compact. On the other hand, for any $w \in V^E$, the product $E^{\times}_w \mathbb{I}^1_E$ is a closed subgroup and the
quotient $\mathbb{I}_E / E^{\times}_w \mathbb{I}^1_E$ is compact (in fact, this quotient is trivial if $w$ is archimedean, and is finite in the function
field case). It follows that for any \emph{nonempty} $T \subset V^E$, the quotient $\mathbb{I}_E(T) / \overline{E^{\times}}^{(T)}$ is compact. Since $S$
contains $V^K_{\infty}$, we conclude that in our notations the quotient $\mathbb{I}_E(\overline{S})/ \overline{E^{\times}}^{(\overline{S})}$ is a profinite
group, hence
$$
\overline{E^{\times}}^{(\overline{S})} = \bigcap B,
$$
where $B$ runs through all open subgroups of $\mathbb{I}_E(\overline{S})$ that contain $E^{\times}$ (note that these automatically have finite
index). Thus, it suffices to show that for any such $B$, the index $[\mathbb{I}_E(\overline{S}) : B]$ divides $[F \colon E \cap F] = [EF : E]$. Let $M$
be the preimage of $B$ under the natural projection $\mathbb{I}_E \to \mathbb{I}_E(\overline{S})$.
By class field theory, for the norm subgroup $N = N_{EF/E}(\mathbb{I}_{EF}) E^{\times}$, the index
$[\mathbb{I}_E : N]$ equals the degree of the maximal abelian subextension of $EF/E$, hence divides
$[EF : E]$. So, it is enough to show that $M$ contains $N$, or equivalently, the abelian extension $P$ of $E$
with the norm subgroup $M$ is contained in $EF$. We note by our construction for every $w \in \overline{S}$, the multiplicative
group $E^{\times}_w$ is contained in $M$, hence the extension $P/E$ splits at $w$ (cf. \cite[Exercise 3]{ANT}).

Let $R$ be the minimal Galois extension of $K$ that contains $E$, $F$ and $P$. Let $\sigma \in \mathscr{C}$ be an automorphism that
acts trivially on $E \cap F$; then there exists $\widetilde{\sigma} \in \Ga(EF/E)$ whose restriction to $F$ is $\sigma$.
We will now show that actually $P \subset (EF)^{\widetilde{\sigma}}$. Assume the contrary. Then there exists $\tau \in \Ga(R/K)$ such
that $\tau \vert EF = \widetilde{\sigma}$ and $\tau \vert P \neq \mathrm{id}_P$. (Indeed, let $\tau_0 \in \Ga(R/K)$ be some lift
of $\widetilde{\sigma}$. If $P \subset EF$ then we can simply take $\tau = \tau_0$. So, suppose $P \not\subset EF$. If every lift
$\tau \in \Ga(R/K)$ of $\widetilde{\sigma}$ acted trivially on $P$, we would have the inclusion $\tau_0 \Ga(R/EF) \subset \Ga(R/P)$. Then
$\Ga(R/EF) \subset \Ga(R/P)$, hence $P \subset EF$, a contradiction. This proves the existence of a required lift $\tau$ in all cases.)
By Tchebotarev's density theorem (cf.\,\cite[Ch. 7, 2.4]{ANT}), there exists a nonarchimedean $v \in V^K \setminus \cP_0$ such that $R$ is unramified
at $v$ and for a suitable extension $u$ we have $\mathrm{Fr}_{R/K}(u \vert v) = \tau$. Clearly, $v \in \cP(F/K , \mathscr{C}) \setminus
\cP_0$, so the restriction $w$ of $u$ to $E$ lies in $\overline{S}$. On the other hand, since $\tau$ restricts to $P$ nontrivially, we see that
$P$ does not split at $w$, a contradiction.
\end{proof}

\medskip

\begin{remark}
The above argument is a modification of the argument given in \cite{PR-Irred} in the case of arithmetic progressions
defined by an abelian extension $F/K$. We note that our argument here shows the index $[\mathbb{I}_E(\overline{S}) : \overline{E^{\times}}^{(\overline{S})}]$
in fact divides the degree $[F^{\sigma} : K]$ for any $\sigma \in \mathscr{C}$ that acts trivially on $E \cap F$ (for this one needs to observe that $[(EF)^{\widetilde{\sigma}}
: E]$ equals $[F^{\sigma} : E \cap F]$, hence divides $[F^{\sigma} : K]$). We also point out that Proposition 4 in \cite{PR-Irred} provides a converse in
the case where $F/K$ is abelian, viz.\:if $\mathscr{C} = \{ \sigma \}$ and $\sigma$ acts on $E \cap F$ nontrivially, then the quotient $\mathbb{I}_E(\overline{S})/
\overline{E^{\times}}^{(\overline{S})}$ has infinite exponent.
\end{remark}

\medskip

Proposition \ref{P:GAP1} gives a form of almost strong approximation property with respect to generalized arithmetic progressions. We now
combine this with the method used in the proof of Proposition \ref{P:WA-tori} to obtain the following.

\begin{thm}\label{T:SA-tori}
{\rm (Almost strong approximation property,  cf.\,\cite[Theorem 3]{PR-Irred})}
For every $d , m \geqslant 1$ there exists an integer $n = n(d , m) \geqslant 1$ such
that given a $K$-torus $T$ of dimension $\leqslant d$, a generalized arithmetic progression
$\cP(F/K , \mathscr{C})$ with $[F : K] = m$ and a finite subset $\cP_0 \subset \cP(F/K ,
\mathscr{C})$, for the set $S = (\cP(F/K , \mathscr{C}) \setminus \cP_0) \cup V^K_{\infty}$, the closure
$\overline{T(K)}^{(S)}$ of $T(K)$ in $T(\mathbb{A}(S))$ contains $T(\mathbb{A}(S))^n$, provided that some (equivalently, every) element of $\mathscr{C}$ acts
trivially on $K_T \cap F$, where $K_T$ is the splitting field of $T$.
\end{thm}
\begin{proof}
Let $n' = n'(d)$ be an integer divisible by the order of any finite subgroup of
the group $\mathrm{GL}_d(\Z)$ (see the proof of Proposition \ref{P:WA-tori}). We will show
that $n(d , m) := n'(d) \cdot m$ is as required. Let $T$ be a $K$-torus of dimension $\leqslant d$
such that for the splitting field $E:= K_T$ some (equivalently, every) element  of $\mathscr{C}$ acts trivially
on $E \cap F$. As in the proof of Proposition \ref{P:WA-tori}, we can construct an exact sequence of $K$-tori
\begin{equation}\label{E:tori1}\tag{$*$}
1 \to T'' \longrightarrow T' \stackrel{\eta}{\longrightarrow} T \to 1
\end{equation}
with $T' = \mathrm{R}_{E/K}(\mathrm{GL}_1)^{\ell}$ for some $\ell \geqslant 1$. Since all the tori in (\ref{E:tori1}) split
over $E$, we have the exact sequence of the groups of $\overline{S}$-adeles, where $\overline{S}$ consists of all extensions
of places from $S$ to $E$:
\begin{equation}\label{E:tori2}\tag{$**$}
1 \to T''(\mathbb{A}_E(\overline{S})) \longrightarrow T'(\mathbb{A}_E(\overline{S})) \stackrel{\eta_{\mathbb{A}_E(\overline{S})}}\longrightarrow T(\mathbb{A}_E(\overline{S})) \to 1.
\end{equation}
Let $\mathcal{G} = \Ga(E/K)$. Then (\ref{E:tori2}) induces the following commutative diagram with  exact bottom row:
$$
\begin{array}{ccccc}
T'(K) & \stackrel{\eta_K}{\longrightarrow} & T(K) & &  \\
\downarrow & & \downarrow &  & \\
T'(\mathbb{A}(S)) & \stackrel{\eta_{\mathbb{A}(S)}}{\longrightarrow} & T(\mathbb{A}(S)) & \longrightarrow & H^1(\mathcal{G} , T''(\mathbb{A}_E(\overline{S})))
\end{array}
$$
Clearly, $\overline{T(K)}^{(S)}$ contains $\eta_{A(S)}(\overline{T'(K)}^{(S)})$. But it follows from Proposition \ref{P:GAP1} that $\overline{T'(K)}^{(S)}$
contains $T'(\mathbb{A}(S))^m$. By the exactness of the bottom row, $T(\mathbb{A}(S))/\eta_{A(S)}(T'(\mathbb{A}(S)))$ has exponent dividing
the order of $\mathcal{G}$, hence $n'(d)$. So, our assertion follows.
\end{proof}

\medskip

\begin{remark}
With some more work, one can prove the following full analog of Proposition \ref{P:GAP1} for arbitrary
tori: {\it There exists $N = N(d , m)$ such that given a $K$-torus $T$ of dimension $\leqslant d$, a generalized arithmetic
progression $\cP(F/K , \mathscr{C})$ with $[F : K] = m$ and a finite $\cP_0 \subset \cP(F/K , \mathscr{C})$, for the set
$S = (\cP(F/K , \mathscr{C}) \setminus \cP_0) \cup V^K_{\infty}$, the index $[T(\mathbb{A}(S)) : \overline{T(K)}^{(S)}]$ is finite and
divides $N$, provided that some $($equivalently, every$)$ element of  \,$\mathscr{C}$ acts trivially on $K_T\cap F$, where $K_T/K$ is the splitting field of $T$. }
Since this more precise statement is not needed in the proof of centrality of the congruence kernel, we will give the details
elsewhere.
\end{remark}

\medskip

\noindent {\it Proof of Theorem B}. The assumption that $S$ almost contains a generalized arithmetic progression $\cP(F/K ,
\mathscr{C})$, of course, means that there exists a finite set  $\cP_0 \subset \cP(F/K , \mathscr{C})$ such that $S$ contains
$S_0 := (\cP(F/K , \mathscr{C}) \setminus \cP_0) \cup V^K_{\infty}$. Besides, we are assuming that $\mathscr{C}$ acts trivially on $F \cap L$,
where $L$ is the minimal Galois extension of $K$ over which $G$ becomes an inner twist of the split group. Since $S_0$ contains a non-archimedean place $v$
such that $G$ is $K_v$-isotropic, our computations of the metaplectic kernel show that $M(S , G) = 1$ (see \cite[Main Theorem]{PR1}). This means that once we know that
$C^{(S)}(G)$ is central, we can actually conclude that it is trivial. We will derive the centrality from Theorem \ref{T:Criterion1}, just as we did in the proof of Theorem
\ref{T:SL}, however the difference is that while almost weak approximation property  holds uniformly for all maximal $K$-tori $T$ of $G$ (see Corollary \ref{C:WA-tori}), Theorem
\ref{T:SA-tori} guarantees almost strong approximation property only in the case where $\mathscr{C}$ acts trivially on $K_T \cap F$. To show that this information is still sufficient
for the proof of centrality, we need the following.
\begin{lemma}\label{L:disjoint} {\rm (Cf.\,\cite[Theorem 2]{PR-Irred-Er})}
Let $G$ be a semi-simple algebraic group  over a global field $K$, and let $L$ be the minimal Galois extension of $K$ over which $G$
becomes an inner form of a split group. Furthermore, suppose we are given a finite subset $\mathcal{S} \subset V^K$ and a finite Galois extension
$F/K$. Then there exists a finite subset $V \subset V^K \setminus \mathcal{S}$ and maximal $K_v$-tori $T(v)$ of $G$ for $v \in V$ such that for any
maximal $K$-torus $T$ of $G$ which is $G(K_v)$-conjugate to $T(v)$, the minimal splitting field $K_T$ satisfies
\begin{equation}\label{E:disjoint}
K_T \cap F = L \cap F.
\end{equation}
\end{lemma}
\begin{proof}
Let $\tau_1, \ldots, \tau_t$ be all the nontrivial elements of $\mathrm{Gal}(F/(F \cap L))$. We extend each $\tau_i$ to $\overline{\tau}_i \in \mathrm{Gal}(FL/K)$ by
letting it act trivially on $L$. There exists a finite subset $V_0$ of $V^K$ such that $G$ is quasi-split over $K_v$
for all $v \in V^K \setminus V_0$ (see \cite[Theorem 6.7]{PlRa}). By Tchebotarev's density theorem \cite[Ch. 7, 2.4]{ANT}, we can find $v_1, \ldots,  v_t \in V^K_f \setminus (S \cup
V_0)$ such that $FL$ is unramified at $v_i$ and for an appropriate extension $w_i \vert v_i$, one has $\mathrm{Fr}_{FL/K}(w_i \vert v_i) = \overline{\tau}_i$, for each
$i = 1, \ldots , t$. Set $V = \{v_1, \ldots , v_t\}$. Since $\overline{\tau}_i$ acts on $L$ trivially, we conclude that $L \subset K_{v_i}$. Combining this with the fact that
by our construction $G$ is quasi-split over $K_{v_i}$, we obtain that $G$ actually splits over $K_{v_i}$, and we let $T(v_i)$ denote its maximal $K_{v_i}$-split torus. We claim that these tori are as required. Indeed, let $T$ be a maximal $K$-torus of $G$ as in the statement of the lemma. Then its splitting field $K_T$ satisfies $K_T \subset K_{v_i}$ for all
$i = 1, \ldots , t$. If we assume that $K_T \cap F \not\subset L \cap F$, then there exists an $i$ such that $\tau_i$ acts nontrivially on $K_T \cap F$. Since $\tau_i = \mathrm{Fr}_{F/K}(v_i)$ lies in the local Galois group $\Ga(FK_{v_i}/K_{v_i})$, we see that $K_T \cap F \not\subset K_{v_i}$. A contradiction, proving the inclusion $\subset$ in
(\ref{E:disjoint}). The opposite inclusion follows from the fact that $L$ is contained in the splitting field of every maximal $K$-torus of $G$.
\end{proof}

\medskip

To implement the above strategy (although with some variations) and prove  Theorem B, we
set $\mathcal{S} = \mathcal{A}(G) \cup V_{\infty}^K$ and use Lemma \ref{L:disjoint} to find a finite subset
$V \subset V^K \setminus \mathcal{S}$ and maximal $K_v$-tori $T(v)$ of $G$ for $v \in V$ with the properties described therein. Then set
$$S' = (\cP(F/K , \mathscr{C}) \setminus (\cP_0 \cup V)) \cup V_{\infty}^K.$$ Clearly, $S'$ is contained in $S$, and in particular is
disjoint from $\mathcal{A}$ and $V$. Now, let $t$ be any regular semi-simple element in $G(K) \cap U$ where $U = \prod_{v \in V} \mathscr{U}(v , T(v))$
in the notations introduced prior to the statement of Theorem \ref{T:Criterion1}, and let $T = Z_G(t)^{\circ}$ be the corresponding maximal $K$-torus of $G$. Then
by construction $T$ is $G(K_v)$-conjugate to $T(v)$ for all $v \in V$, so by Lemma \ref{L:disjoint} we have $K_T \cap F = F \cap L$. This means that the elements of
$\mathscr{C}$ act trivially  on $K_T \cap F$, and therefore  Theorem \ref{T:SA-tori} yields the inclusion $\overline{T(K)}^{(S')} \supset T(\mathbb{A}(S'))^n$ where $n = n(d , [F:K])$ is the number from this theorem and $d$
is the absolute rank of $G$.  On the other hand, we obviously have the inclusion $\pi^{(S')}(Z_{\widehat{G}^{(S')}}(t)) \supset \overline{T(K)}^{(S')}$. This verifies the assumptions of Theorem \ref{T:Criterion1}(ii) for the congruence sequence (\ref{E:C}) associated with the set $S'$, and therefore enables us to conclude that $C^{(S')}$
is central. As we explained in the beginning of the proof, since there exists a nonarchimedean $v \in S'$ such that $G$ is $K_v$-isotropic, this implies that
$C^{(S')}(G)$ is actually trivial. Finally, since $S' \subset S$ and $S \setminus S'$ does not contain any anisotropic places for $G$, there exists a natural \emph{surjective}
homomorphism $C^{(S')}(G) \to C^{(S)}(G)$ (cf.\,\cite[Lemma 6.2]{Ra1}), so $C^{(S)}(G)$ is also trivial. \hfill $\Box$.

\medskip

\section{CSP for arithmetic groups with adelic profinite
completion: Proof of Theorem c}\label{S:Adelic}

\setcounter{equation}{0}


Before we embark on the proof of Theorem C (of the introduction), we would like to point out that for infinite
arithmetic groups in positive characteristic the profinite completion is \emph{never}
adelic (see Remark \ref{R:adelic}  below), so we limited the statement of Theorem C to the case of number fields.
The proof relies on the following properties of
the group of adeles.
\begin{lemma}\label{L:adelic1}
Let $\displaystyle \Omega = \mathrm{GL}_n(\widehat{\mathbb Z})  = \prod_{q \ {\rm prime}}
\mathrm{GL}_n({\mathbb Z}_q),$ and fix a prime $p$.

\medskip

\noindent {\rm (1)} \parbox[t]{15.7cm}{There exists $d \geqslant 1$
(depending only on $n$) such that
for any pro-$p$ subgroup ${\mathcal P}$ of $\Omega$, one has
$$[{\mathcal P}^{(d)} , {\mathcal P}^{(d)}] \subset
{\rm{GL}}_n({\mathbb Z}_p),
$$
where ${\mathcal P}^{(d)}$ denotes the (closed)
subgroup generated by the $d$-th powers of elements of ${\mathcal P.}$}

\medskip

\noindent {\rm (2)} \parbox[t]{15.7cm}{If ${\mathcal P} \subset \Omega$ is
an analytic pro-$p$ subgroup satisfying the following condition \vskip2mm
{\rm (O)} \hskip1mm \parbox[t]{12cm}{for any open subgroup ${\mathcal P}' \subset
{\mathcal P}$, the commutator subgroup $[{\mathcal P}' , {\mathcal P}']$
 is also open in ${\mathcal P}$, }
 \vskip1mm

 \parbox[t]{12cm}{then the kernel of the
projection $ {\mathcal P} \to \prod_{q \neq p}
{\rm{GL}}_n({\mathbb Z}_q)$ is open in ${\mathcal P}.$}}
\end{lemma}
\begin{proof} (1): By Jordan's Theorem
(cf., for example, \cite{D}), there exists
$\ell = \ell(n)$ such that every finite subgroup $J \subset \mathrm{GL}_n(F),$ where
$F$ is a field of characteristic zero, contains an abelian normal
subgroup of index $\leqslant \ell$. Set $d = \ell !$ and observe that the exponent
of any group of order $\leqslant \ell$ divides $d$. For a prime $q$, we let
$\mathrm{pr}_{q} \colon \Omega \to \mathrm{GL}_n({\mathbb Z}_{q})$ denote the
corresponding projection. The first congruence subgroup $\mathrm{GL}_n(\mathbb{Z}_q , q)$ is
a normal pro-$q$ subgroup of $\mathrm{GL}_n(\mathbb{Z}_q)$ of finite index. This
means that for any $q \neq p$, the image $\mathrm{pr}_q(\cP)$ has trivial intersection with
$\mathrm{GL}_n(\mathbb{Z}_q , q)$, hence is finite. Then our choice of $d$ forces $\mathrm{pr}_q(\cP^{(d)}) =
\mathrm{pr}_q(\cP)^{(d)}$ to be abelian, implying that $\mathrm{pr}_q([\cP^{(d)} , \cP^{(d)}])$
is trivial. This being true for all $q \neq p$, we obtain that $[\cP^{(d)} , \cP^{(d)}]$ is contained
in $\mathrm{GL}_n(\mathbb{Z}_p)$, as asserted.

\medskip

(2): Let $d$ be the integer from part (1).
Since ${\mathcal P}$ is analytic, it follows from the Implicit Function Theorem that the map
$\cP \to \cP$, $x \mapsto x^d$, is open, and therefore $\cP^{(d)}$ is an open
subgroup of $\cP$. (We note that as follows from the affirmative solution of the Restricted Burnside
Problem by E.I.~Zelmanov \cite{Zel1}, \cite{Zel2}, the subgroup $\cP^{(d)}$ is open in $\cP$ for
\emph{any finitely generated} profinite group $\cP$ and any integer $d \geqslant 1$, so the assumption of
analyticity here can be replaced by just requiring finite generation.)  Then, due to assumption (O),
the commutator subgroup $[{\mathcal P}^{(d)} , {\mathcal P}^{(d)}]$ is also open in ${\mathcal P}$.
On the other hand, part (1) asserts that $[{\mathcal P}^{(d)} ,
{\mathcal P}^{(d)}]$ is contained in the kernel of the projection
$\displaystyle {\mathcal P} \to \prod_{q \neq p}
\mathrm{GL}_n({\mathbb Z}_q)$, which therefore  is open in ${\mathcal P}$.
\end{proof}

\medskip

\begin{remark}\label{R:adelic}
Combining Lemma \ref{L:adelic1} with the results of \cite{JaK} (we thank M.\:Ershov for this reference), one shows that
for an $S$-arithmetic subgroup $\Gamma$ of an absolutely almost simple simply connected algebraic group $G$ over a global
field $K$ of characteristic $p > 0$, the profinite completion $\widehat{\Gamma}$ is not adelic provided that that $\mathrm{rk}_S\: G > 0$
(i.e. $\Gamma$ is infinite) and $S \neq V^K$. Indeed, pick $v \in V^K \setminus S$ and let $P = G(\mathcal{O}_v , \mathfrak{p}_v)$ be the
congruence subgroup modulo the valuation ideal $\mathfrak{p}_v$ of the valuation ring $\mathcal{O}_v \subset K_v$. We will view $P$ as a
pro-$p$ subgroup of $\overline{\Gamma}$, and let $\mathcal{P}$ be a Sylow pro-$p$ subgroup of $\pi^{-1}(P)$, so that $\pi(\mathcal{P}) = P$. Assume that
there exists an embedding $\widehat{\Gamma} \hookrightarrow \mathrm{GL}_n(\widehat{\Z})$. The according to Lemma \ref{L:adelic1}, for some $d \geqslant 1$,
the subgroup $\mathcal{P}_0 = [\mathcal{P}^{(d)} , \mathcal{P}^{(d)}]$ is contained in $\mathrm{GL}_n(\Z_p)$, hence is a $p$-adic analytic group. Then
$P_0 := \pi(\mathcal{P}_0)$ is also $p$-adic analytic. On the other hand, it follows from \cite[Theorem 1.7]{JaK} or from \cite{Ri} that $P_0$ is an open
subgroup of $P$, and therefore  cannot be analytic (see \cite[Theorem 1.5]{JaK} or \cite[Theorem 13.23]{Anal}). A contradiction.
\end{remark}

\medskip

To proceed with the proof of Theorem C, for a prime $p$ we let $V(p)$ denote the finite
set $\{ v \in V^K \setminus S \: \vert \: v(p) \neq 0 \}$, and let $\Pi$ be the finite set of primes $p$
for which $V(p) \cap \mathcal{A}(G) \neq \varnothing$. To prove Theorem C, it is enough to show that
$$
V_0 := \bigcup_{p \notin \Pi} V(p)
$$
is contained in $\mathfrak{Z} = \mathfrak{Z}(C^{(S)}(G))$. Indeed, since the complement $V^K \setminus (S \cup V_0)$
is finite, by Theorem~\ref{T:SL} the congruence kernel $C^{(S \cup V_0)}(G)$ is trivial. So, the inclusion $V_0 \subset \mathfrak{Z}$
would enable us to derive the centrality of $C^{(S)}(G)$ from Proposition \ref{P:Centr1}.

The rest of the argument focuses on proving the inclusion $V(p) \subset \mathfrak{Z}$ for a fixed $p \notin \Pi$.
By our assumption there exists a continuous embedding $\iota \colon \widehat{\Gamma} \hookrightarrow \prod_q \mathrm{GL}_n(\Z_q)$. Then
$$
\cP := \iota^{-1}(\mathrm{GL}_n(\Z_p , p))
$$
is an analytic pro-$p$ normal subgroup of $\widehat{\Gamma}$.
\begin{lemma}\label{L:adelic2}
$\pi(\cP)$ contains an open subgroup of $G_{V(p)} = \prod_{v \in V(p)} G(K_v)$.
\end{lemma}
\begin{proof}
Let $\mathcal{S}_p$ be a Sylow pro-$p$ subgroup of $\widehat{\Gamma}$ (cf., for example, \cite[Ch.\,I, \S 1.5]{Serre2}).  Then $\pi(\mathcal{S}_p)$ is a Sylow pro-$p$ subgroup
of $\overline{\Gamma} = \prod_{v \notin S} G(\mathcal{O}_v)$ ({\it loc.\,cit.,} Prop.\:4). Since for every $v \in V(p)$, the congruence subgroup $G(\mathcal{O}_v , \mathfrak{p}_v)$ modulo
the maximal ideal $\mathfrak{p}_v$ of $\mathcal{O}_v$ is a normal pro-$p$ subgroup of $G(\mathcal{O}_v)$, the conjugacy theorem for Sylow pro-$p$
subgroups of profinite groups ({\it loc.\,cit.}, Prop.~3) implies that
$$
\prod_{v \in V(p)} G(\mathcal{O}_v , \mathfrak{p}_v) \, \subset \, \pi(\mathcal{S}_p),
$$
and consequently
\begin{equation}\label{E:adelic7}
\prod_{v \in V(p)} \left[G(\mathcal{O}_v , \mathfrak{p}_v)^{(d)} \ , \ G(\mathcal{O}_v , \mathfrak{p}_v)^{(d)} \right] \, \subset \, \pi\left(\left[ \mathcal{S}_p^{(d)} \ , \
\mathcal{S}_p^{(d)} \right] \right)
\end{equation}
for any $d \geqslant 1$. We will use this for the integer $d$ given by Lemma \ref{L:adelic1}(1). Then $\iota\big(\big[ \mathcal{S}^{(d)}_p , \mathcal{S}^{(d)}_p \big]\big)$ is contained in $\mathrm{GL}_n(\Z_p)$. So, the intersection $\cP \cap \big[ \mathcal{S}^{(d)}_p , \mathcal{S}^{(d)}_p \big]$ is of finite index in $\big[ \mathcal{S}^{(d)}_p , \mathcal{S}^{(d)}_p \big]$, and therefore $\pi(\cP) \cap \pi\big(\big[ \mathcal{S}_p^{(d)} , \mathcal{S}_p^{(d)} \big]\big)$ is of finite index in $\pi\big(\big[ \mathcal{S}_p^{(d)} , \mathcal{S}_p^{(d)} \big]\big)$.
On the other hand, as in the proof of Lemma \ref{L:adelic1}(1), the map $G(K_v) \to G(K_v)$, $x \mapsto x^d$, is open, making $G(\mathcal{O}_v , \mathfrak{p}_v)^{(d)}$ an open
subgroup of $G(K_v)$. Furthermore, since $G$ is an absolutely almost simple group, its Lie algebra (as an analytic group over $\Q_p$) is semi-simple, which by way of the Implicit Function Theorem implies that the commutator subgroup of any open subgroup of $G(K_v)$ is again open (cf. \cite{Ri}). Combining these two facts, we see that the left-hand side of (\ref{E:adelic7}) in open in $G_{V(p)}$. Then $\pi(\cP)$ is also open, as required.
\end{proof}

\medskip

Since $\cP$ is an analytic pro-$p$ group, it follows from Cartan's theorem (cf. \cite[ch. III, \S8, n$^{\circ}$~2]{Bour}, \cite{Anal}) and the preceding  lemma that $\mathscr{U} := \cP
\cap \pi^{-1}(G_{V(p)})$ is also an analytic pro-$p$ normal subgroup of ${\widehat{\Gamma}}$ having the property that $\pi(\mathscr{U}) \subset G_{V(p)}$ is open. The latter means that
for the $\Q_p$-Lie algebras $\mathfrak{u}$ and $\mathfrak{g}$ of $\mathscr{U}$ and $G_{V(p)}$ respectively (as analytic pro-$p$ groups), and for the
differential of $\pi$ we have
\begin{equation}\label{E:adelic8}
d\pi(\mathfrak{u}) = \mathfrak{g}.
\end{equation}
The rest of the proof relies on the analysis of conjugates $g \mathscr{U} g^{-1}$ for $g \in \widehat{G}$. This analysis, however, is complicated by the fact
that $\mathscr{U}$ may not satisfy condition (O) of Lemma \ref{L:adelic1}(2). To bypass this difficulty, we first replace $\mathscr{U}$ with a smaller subgroup
that satisfies this condition and retains other significant properties of $\mathscr{U}$. More precisely, let
$$
\mathfrak{u}_0 = \mathfrak{u}, \ \ \ \mathfrak{u}_{i+1} = [\mathfrak{u}_i , \mathfrak{u}_i] \ \ \text{for} \ \ i \geqslant 0
$$
be the derived series of $\mathfrak{u}$. Pick $\ell \geqslant 0$ so that $\mathfrak{u}_{\ell + 1} = \mathfrak{u}_{\ell}$, set $\mathfrak{w} =
\mathfrak{u}_{\ell}$ and let $\mathcal{W} \subset \mathscr{U}$ be a closed subgroup with the Lie algebra $\mathfrak{w}$. Since by construction
$[\mathfrak{w} , \mathfrak{w}] = \mathfrak{w}$, the subgroup $\mathcal{W}$ satisfies (O).
At the same time, it follows from (\ref{E:adelic8})
and our construction that $d\pi(\mathfrak{w}) = \mathfrak{g}$, and therefore $\pi(\mathcal{W})$ is open in $G_{V(p)}$.
\begin{lemma}\label{L:adelic3}
For any $g \in \widehat{G}$, the subgroups $\mathcal{W}$ and $g\mathcal{W}g^{-1}$ are commensurable.
\end{lemma}
\begin{proof}
First, note that both $\widehat{\Gamma}$ and $g \widehat{\Gamma} g^{-1}$ are open compact subgroups of $\widehat{G}$, hence are commensurable. It follows
that there exists an open subgroup  $\mathcal{W}' \subset \mathcal{W}$ such that $\widetilde{\mathcal{W}} := g\mathcal{W}' g^{-1} \subset \widehat{\Gamma}$. Since $\mathcal{W}$, hence also $\mathcal{W}'$,  satisfies condition (O), Lemma \ref{L:adelic1}(2) tells us that after replacing $\mathcal{W}'$ with a smaller open subgroup we may assume that $\iota(\widetilde{\mathcal{W}})$ is contained in $\mathrm{GL}_n(\Z_p)$ and even in $\mathrm{GL}_n(\Z_p , p)$, i.e. $\widetilde{\mathcal{W}} \subset \cP$. At the same time, since $G_{V(p)}$ is normal in $\overline{G}$, we see that $\pi(\widetilde{\mathcal{W}}) \subset G_{V(p)}$, and eventually
$\widetilde{\mathcal{W}} \subset \cP \cap \pi^{-1}(G_{V(p)}) =
\mathscr{U}$. The Lie algebra $\widetilde{\mathfrak{w}}$ is isomorphic to $\mathfrak{w}$, and hence is its own commutator. It follows that $\widetilde{\mathfrak{w}}$ is contained in the
$\ell$th term of the derived series $\mathfrak{u}_{\ell} = \mathfrak{w}$, and therefore $\widetilde{\mathfrak{w}} = \mathfrak{w}$. But since $\widetilde{\mathcal{W}}$
and $\mathcal{W}$ are both closed subgroups of the analytic pro-$p$ group $\mathscr{U}$, the fact that they have the same Lie algebras means that they share an open subgroup,
hence are commensurable, and our assertion follows.
\end{proof}

\medskip

For an arbitrary $g \in \widehat{G}$, the corresponding inner automorphism $\mathrm{Int}\: g$ induces a continuous group homomorphism
$g^{-1} \mathcal{W}g \to \mathcal{W}$ of analytic pro-$p$ groups, which is then analytic. It follows from Lemma \ref{L:adelic3} that both groups
have the same Lie algebra $\mathfrak{w}$, so we obtain an action of $g$ on the latter. Furthermore, using the fact that for any $g_1 , g_2 \in \widehat{G}$,
all four subgroups $\mathcal{W}$, $g^{-1}_1 \mathcal{W} g_1$, $g^{-1}_2 \mathcal{W} g_2$ and $(g_1g_2)^{-1} \mathcal{W} (g_1g_2)$ are pairwise commensurable,
it is easy to see that in fact we obtain a continuous representation $\rho \colon \widehat{G} \to \mathrm{GL}(\mathfrak{w})$.

\begin{lemma}\label{L:adelic4}
For $C = C^{(S)}(G)$, the image $\rho(C)$ is finite.
\end{lemma}
\begin{proof}
Since $C$ is compact, the image $\rho(C)$  is a compact subgroup of $\mathrm{GL}(\mathfrak{w}) = \mathrm{GL}_m(\Q_p)$ where
$m = \dim_{\Q_p} \mathfrak{w}$. Since $\mathrm{GL}_m(\Q_p)$ is a $p$-adic analytic group, we conclude that $\rho(C)$ is finitely generated (cf. \cite{Anal}). Now,
applying Proposition \ref{P:Prop-F} to $F = C/(C \cap \mathrm{Ker}\: \rho)$, we see that $\rho(C)$ is finite.
\end{proof}

\medskip

Let $C_0 := C \cap \mathrm{Ker}\: \rho$ is an open normal subgroup of $C$ normalized by $\widehat{G}$. Then the conjugation action of $C_0$ on $\mathcal{W}$
induces the trivial action on the Lie algebra $\mathfrak{w}$. This means that we can replace $\mathcal{W}$ with an open subgroup to ensure that $C_0$ centralizes $\mathcal{W}$ (we note that after this replacement, the image $\overline{\mathcal{W}} := \pi(\mathcal{W})$ will still be open in $G_{V(p)}$).
\begin{lemma}\label{L:adelic5}
There exists $g \in G_{V(p)}$ such that $\overline{\mathcal{W}}$ and $g\overline{\mathcal{W}}g^{-1}$ generate $G_{V(p)}$.
\end{lemma}
\begin{proof}
It is enough to show  that given $v \in V(P)$ and an open subgroup $W$ of $G(K_v)$, there exists $g \in G(K_v)$ such that
\begin{equation}\label{E:adelic100}
G(K_v) = \langle W , gWg^{-1} \rangle.
\end{equation}
Note that $G(K_v)$ has only finitely open compact subgroup $W_1, \ldots , W_r$ that contain $W$ (cf. \cite[Proposition 3.16]{PlRa}). Pick
a regular semi-simple element $t \in W$. It is well-known that the conjugacy class  of $t$ in $G(K_v)$ is closed and non-compact. So, one can
find $g \in G(K_v)$ such that
$$
gtg^{-1} \notin \bigcup_{i = 1}^r W_i.
$$
Then this $g$ is as required. Indeed, in this case the right-hand side of (\ref{E:adelic100}) is an open non-compact subgroup,
and therefore by a theorem due to Tits (see \cite{Pr-RT}) contains $G(K_v)^+$, the subgroup generated by the $K_v$-points of the $K_v$-defined
parabolics of $G$. But since $G$ is simply connected and $v \notin \mathcal{A}(G)$, we have $G(K_v)^+ = G(K_v)$ (see the discussion in the beginning
of \S\ref{S:ThD}), and (\ref{E:adelic100}) follows.
\end{proof}

Now, let $g \in G(K_v)$ be as in Lemma \ref{L:adelic5}, pick a lift $\widehat{g} \in \pi^{-1}(g)$, and set $C_1 = C_0 \cap \widehat{g} C_0 \widehat{g}^{-1}$. Clearly, $C_1$ is an open normal subgroup of $C$ that is centralized by $\mathcal{W}$ and $\widehat{g} \mathcal{W} \widehat{g}^{-1}$. So, if we let $\mathcal{G} = \pi^{-1}(G_{V(p)})$ and $\mathcal{Z} = Z_{\mathcal{G}}(C_1)$, then it follows from our construction that $\pi(\mathcal{Z}) = G_{V(p)}$. Let $\mathcal{Z}_1$ denote the kernel of the natural action of $\mathcal{Z}$ on the
finite group $C/C_1$. Since $G_{V(p)}$ does not have proper normal subgroups of finite index, we will still have $\pi(\mathcal{Z}_1) = G_{V(p)}$. Then as in Lemma \ref{L:A2}, for any $c \in C$, the map $x \mapsto [c , x]$ defines a continuous group homomorphism $\chi_c \colon \mathcal{Z}_1 \to C_1$, and we can consider
$$
\chi \colon \mathcal{Z}_1 \longrightarrow \mathcal{C} := \prod_{c \in C} X_c,  \ \ \text{where} \ \ X_c = C_1 \ \ \text{for all} \ \ c \in C,
$$
given by $\chi(x) = (\chi_c(x))$. It follows from Lemma \ref{L:A1} that for $\mathcal{H} := \mathrm{Ker}\: \chi$, we have $\pi(\mathcal{H}) = G_{V(p)}$. At the same time,
by our construction, $\mathcal{H} \subset Z_{\widehat{G}}(C)$, which implies that $V(p) \subset \mathfrak{Z}$, as required. This completes the proof of
Theorem C.

\section{Generators for the congruence kernel: Proof of Theorem D}\label{S:ThD}

In this section we will assume that $\mathrm{char}\: K = 0$, and let $G$ be an absolutely almost
simple simply connected $K$-{\it isotropic} algebraic group. If a subset $S \subset V^K$ containing $V_{\infty}^K$
is such that $\mathrm{rk}_S\: G \geqslant 2$ then according to the results of Raghunathan \cite{Ra1}, \cite{Ra2}
that in particular prove Serre's conjecture in the isotropic situation, the congruence kernel $C^{(S)}(G)$ is central,
hence is isomorphic to the metaplectic kernel $M(S , G)$, which in all cases is a finite cyclic group (often trivial). In the remaining
case where $\mathrm{rk}_S\: G = 1$ (and therefore necessarily $\mathrm{rk}_K\: G = 1$ and $\vert S \vert = 1$, hence $K$ is either $\Q$ or
an imaginary quadratic field), according to Serre's conjecture
$C^{(S)}(G)$ is expected to be infinite, which has been established in a number of cases although we do not yet have a general result.
The goal of this section is to provide several convenient systems of generators (or rather almost generators)
for $C^{(S)}(G)$ \emph{as a normal subgroup} of $\widehat{G}^{(S)}$ and
eventually reduce one of them to a \emph{single} element, proving thereby Theorem D (we recall that according to Proposition \ref{P:Prop-F}, if $C^{(S)}(G)$
is infinite, it cannot be finitely generated as a~group). So, henceforth we will assume that $\mathrm{rk}_K\: G = 1$.

First, we need to fix some notations that will be kept throughout this section. Let $T$ be a maximal $K$-split torus of $G$ (so, $\dim T = 1$), and $M =
Z_G(T)$. The root system $\Phi = \Phi(G , T)$ is either $\{ \pm \alpha \}$ or $\{ \pm \alpha, \pm 2\alpha \}$. For $\beta \in \Phi$, we let $U_{\beta}$ denote
the corresponding unipotent $K$-subgroup of $G$ (cf. \cite[21.9]{Bo}, \cite[\S5]{BT}, \cite[15.4]{Spr}); recall that $U_{\pm 2\alpha} \subset U_{\pm \alpha}$ if $2\alpha \in \Phi$; if $2\alpha\notin \Phi$,
then $U_{\pm 2\alpha}$ will denote the trivial subgroup of $U_{\alpha}$. The subgroups  $P_{\pm \alpha} =
M \cdot U_{\pm \alpha}$ (semi-direct product) are opposite minimal parabolic $K$-subgroups with the unipotent radicals $U_{\alpha}$ and $U_{-\alpha}$
respectively and the common Levi subgroup $M = P_{\alpha} \cap P_{-\alpha}$. Following Tits \cite{T-simple}, for a field extension $F/K$ we let $G(F)^+$
denote the subgroup of $G(F)$ generated by the $F$-rational points of the unipotent radicals of parabolic $F$-subgroups (since $\mathrm{char}\: K = 0$,
it is simply the subgroup generated by all unipotent elements of $G(F)$). It is known \cite[6.2(v)]{BoT-Hom} that $G(F)^+$ is generated by $U_{\alpha}(F)$
and $U_{-\alpha}(F)$. On the other hand, from the affirmative solution of the Kneser-Tits problem over local  (see \cite{Pl-SA}, \cite{Pr-Rag},  \cite[\S 7.2]{PlRa}) and
global (see \cite{Gille}) fields, one knows that $G(K)^+ = G(K)$ and  $G(K_v)^+ = G(K_v)$ for any $v \in V^K$. Thus,
$U_{\alpha}(K)$ and $U_{-\alpha}(K)$  generate $G(K)$ and $U_{\alpha}(K_v)$ and $U_{-\alpha}(K_v)$ generate $G(K_v)$ for any $v$.

\medskip

We will now produce the first generating system for $C = C^{(S)}(G)$ as a normal subgroup of $\widehat{G}^{(S)}$ by generalizing the construction used in Examples 4.6 and 4.7.
Since $\mathrm{char}\: K = 0$, the topologies $\tau_a$ and $\tau_c$ of $G(K)$ induce the same topology on $U_{\pm \alpha}(K)$ (cf.\,\cite[Prop.\,2.1]{Ra3}). It follows that $\pi^{(S)}$ induces isomorphisms
$$
\widehat{U_{\pm \alpha}(K)} \longrightarrow \overline{U_{\pm \alpha}(K)} = U_{\pm \alpha}(\mathbb{A}(S)),
$$
and we let $\sigma_{\pm \alpha} \colon \overline{U(K)} \to \widehat{U_{\pm \alpha}(K)}$ denote the inverse isomorphisms. Consider the set
$$
X = \bigcup X(v_1 , v_2), \ \ \ with \ \ X(v_1 , v_2) :=  \left[\sigma_{\alpha}(U_{\alpha}(K_{v_1})) \: , \: \sigma_{-\alpha}(U_{-\alpha}(K_{v_2}))\right],
$$
where the union is taken over all $v_1 , v_2 \in V^K \setminus S$, $v_1 \neq v_2$, and $[A , B]$ denotes the set of all commutators $[a , b]$ with
$a \in A$, $b \in B$. The fact that  the groups $G(K_{v_1})$ and $G(K_{v_2})$ for such $v_1, v_2$ commute elementwise inside $\overline{G}^{(S)} = G(\mathbb{A}(S))$ immediately
implies that $X(v_1 , v_2) \subset C$, hence $X \subset C$. Now, let $D$ be the closed normal subgroup of $\widehat{G}^{(S)}$ generated by $X$ and consider the corresponding extension (\ref{E:F}) of \S 2. For $v \in V^K \setminus S$, we let $H_v$ denote the image in $H = \widehat{G}^{(S)}/D$ of the subgroup $\mathcal{G}_v \subset \widehat{G}^{(S)}$ generated by $\sigma_{\alpha}(U_{\alpha}(K_v))$ and $\sigma_{-\alpha}(U_{-\alpha}(K_v))$. As we mentioned above, $G(K_v) = \left\langle U_{\alpha}(K_v) , U_{-\alpha}(K_v) \right\rangle$, which implies that $\theta(H_v) = G(K_v)$. Furthermore, by our construction the subgroups $H_{v_1}$ and $H_{v_2}$ commute elementwise in $H$. Finally, the closed subgroup of $\widehat{G}^{(S)}$ generated by the
$\mathcal{G}_v$'s for $v \in V^K \setminus S$ will contain $\widehat{U_{\alpha}(K)}$ and $\widehat{U_{-\alpha}(K)}$, hence $G(K) = \left\langle U_{\alpha}(K) , U_{-\alpha}(K)
\right\rangle$, hence coincides with $\widehat{G}^{(S)}$. Now, applying Proposition \ref{P:Com-lifts2} to the partition of $V^K \setminus S$ into singletons and the subgroups $H_v$ constructed above, we obtain that (\ref{E:F}) is a central extension. Thus, $F = C/D$ is a quotient of the metaplectic kernel $M(S , G)$, hence it is  a finite cyclic group of order
dividing the order $|\mu_K|$ of the group $\mu_K$ of roots of unity in $K$ (cf.\,\cite{PR1}).

\medskip

Next, we will show that by using the result of Raghunathan \cite{Ra2}, that this system can be substantially reduced.
\begin{prop}\label{P:Gen-CK1}
Fix $v_0 \in V^K \setminus S,$ and set
$$
Y(v_0) = \bigcup_{v \in V^K \setminus (S \cup \{ v_0 \})} X(v_0 , v).
$$
Then $Y(v_0) \subset C$, and if $D$ is the closed normal subgroup of $\widehat{G}^{(S)}$ generated
by $Y(v_0)$, then $C/D$ is a quotient of $M(S , G)$, hence it is a finite cyclic group of order dividing $\vert \mu_K \vert$.
\end{prop}
\begin{proof}
The discussion above yields the inclusion $Y(v_0) \subset C$ and also shows that it is enough  to prove that the corresponding sequence (\ref{E:F})
is a central extension. Since there exists $\omega \in G(K)$ such that $\omega U_{\pm \alpha} \omega^{-1} = U_{\mp \alpha}$ (cf. \cite[21.2]{Bo}, \cite[5.3]{BT}), the group $D$
contains $[\sigma_{-\alpha}(U(K_{v_0})) , \sigma_{\alpha}(U(K_v))]$ for any $v \in V^K \setminus (S \cup \{ v_0 \})$. We will now use Proposition
\ref{P:Com-lifts1} to establish the centrality. Write  $V^K \setminus S = V_1 \cup V_2$ where $V_1 = \{ v_0 \}$ and $V_2 = V^K \setminus (S \cup \{ v_0 \})$
(obviously, $\mathcal{A} = \varnothing$ in our situation); then $V'_i = V_{3-i}$. The congruence kernel $C^{(S \cup V'_1)}(G) = C^{(V^K \setminus \{ v_0 \})}(G)$
is trivial by Theorem \ref{T:SL}, and the congruence kernel $C^{(S \cup V'_2)}(G) = C^{(S \cup \{v_0\})}(G)$ is central by \cite{Ra2}, hence is isomorphic to $M(S \cup
\{ v_0 \} , G)$, which is trivial \cite{PR1}, and thus is trivial as well.  Let $H_1$ be the closed subgroup of $H$ generated by the images of
$\sigma_{\alpha}(U_{\alpha}(K_{v_0}))$ and $\sigma_{-\alpha}(U_{-\alpha}(K_{v_0}))$, and let $H_2$ be the closed subgroup generated by the images of $\sigma_{\alpha}(U_{\alpha}(K_v))$ and $\sigma_{-\alpha}(U_{-\alpha}(K_v))$ for $v \in V^K \setminus (S \cup \{ v_0 \})$ (or, equivalently, by the images of
$\sigma_{\alpha}(U_{\alpha}(\mathbb{A}(S \cup \{ v_0 \})))$ and $\sigma_{-\alpha}(U_{-\alpha}(\mathbb{A}(S \cup \{ v_0 \})))$). Then $\theta(H_i) = G(\mathbb{A}(S
\cup V'_i))$ for $i = 1, 2$, and moreover, $H_1$ and $H_2$ commute elementwise and together generate a dense subgroup of $H$. In other words, $H_1$ and $H_2$ satisfy
the assumptions of Proposition \ref{P:Com-lifts1}, and then the required centrality of (\ref{E:F}) immediately follows from this proposition.
\end{proof}

\medskip

We will first establish Theorem D for $G = \mathrm{SL}_2$ where the (almost) generating element $c$ can be written down explicitly.
The argument here is inspired by the proof of Proposition \ref{P:Gen-CK1}, but relies
only on the result of Example 4.7 (which is originally due to Serre \cite{Serre1}) rather than on the general result of Raghunathan \cite{Ra2}.
We will keep the notations introduced in Example 4.7. Fix $v_0 \in V^K \setminus S$, write $\mathbb{A}(S) = K_{v_0} \times \mathbb{A}(S \cup \{ v_0 \})$ and
consider the elements
$$
1_{v_0} \in  K_{v_0} \ \ \text{and} \ \ 1'_{v_0} = (1, \ldots, 1, \ldots) \in \mathbb{A}(S \cup \{ v_0 \}).
$$
\begin{prop}\label{P:SL2-gen}
Set $c(v_0) = [\widehat{u}^+(1_{v_0}) \, , \widehat{u}^-(1'_{v_0})] \in C$, and let $D$ be the closed normal subgroup of $\widehat{G}$ generated by $c(v_0)$. Then
the quotient $C/D$ is central in $\widehat{G}/D$, hence it is a finite cyclic group of order dividing $\vert \mu_K \vert$.
\end{prop}
\begin{proof}
First, we observe that the set
$$
\Delta = \{ (t^{-1} , t) \in K_{v_0} \times \A(S \cup \{ v_0 \}) \
\vert \ t \in K^{\times} \}
$$
is dense in $\A(S).$ Indeed, any open set in $\A(S)$ contains an
open set of the form $U^{-1}_{v_0} \times U'_{v_0}$ for some open
sets $U_{v_0} \subset K_{v_0}^{\times}$ and $U'_{v_0} \subset \A(S
\cup \{ v_0 \})$, and then our claim immediately follows from the density
of $K$ in $\mathbb{A}(S)$ (strong approximation with respect to $S$).
Since $$h(t)^{-1}c(v_0)h(t) =
[\widehat{u}^+(t^{-2}) , \widehat{u}^-(t^2)],$$ we obtain that $D$ contains
the set
$$
\{[\widehat{u}^+(a) , \widehat{u}^-(b)] \ \vert \ a \in K^2_{v_0}, \ b \in
\A(S \cup \{ v_0 \})^2 \}.
$$
In particular, for any $v \in V^K \setminus (S \cup \{v_0\})$, all commutators
$$
[\widehat{u}^+(a) , \widehat{u}^-(b)] \ \ \text{with} \ \ a \in K^2_{v_0}, \ b \in K^2_{v}
$$
lie in $D$. Since for any $w \in V^K_f$, every element of $K_w$ can be written as a sum of (at most four)
squares, the identities
$$
[xy , z] = (x[y , z]x^{-1})[x , z] \ \ \ \text{and} \ \ \ [x , yz] =
[x , y](y[x , z]y^{-1})
$$
imply that $D$ in fact contains the set $X(v_0 , v)$. Then $D$ contains $Y(v_0)$,
and our claim follows from Proposition \ref{P:Gen-CK1}.
\end{proof}

\begin{remark}\label{R:order}
As we already observed, if the group $G$ is $K$-isotropic then $\mathrm{rk}_S G = 1$ is possible only if $K = \Q$
or $K = \Q(\sqrt{-d})$, $d$ square-free $>0$, with $S$ consisting of the unique archimedean place in both cases. If $K = \Q$ then
$M(S , G)$ for any $G$ is of order $\leqslant 2$, and in fact $M(S , G)$ is trivial for $G = \mathrm{SL}_2$. The latter means that the congruence
kernel for $\mathrm{SL}_2(\Z)$ is generated as a normal subgroup of $\widehat{G}$ by the element $c(p)$ constructed above for any prime $p$.
On the other hand, for $K = \Q(\sqrt{-d})$, the order of $M(S , G)$, hence also that of $C/D$, divides 2 (resp., 4 and 6) if $d \neq 1, 3$ (resp.,
$d = 1$ and $d = 3$).
\end{remark}

\medskip

It is worth mentioning that the construction of generators described in Proposition \ref{P:SL2-gen} has some other applications. Let $G_0 = \mathrm{SL}_2$
over $\Q$, and let $S_0 = V_{\infty}^{\Q}$ so that $\Gamma_0 = G(\mathcal{O}_{\Q}(S_0))$ is $\mathrm{SL}_2(\Z)$. Furthermore, fix a square-free integer $d > 0$, let
$G_d = \mathrm{SL}_2$ over $K_d := \Q(\sqrt{-d})$ and $S_d = V_{\infty}^{K_d}$ so that $\Gamma_d = G(\mathcal{O}_{K_d}(S_d))$ is the Bianchi group $\mathrm{SL}_2(\mathcal{O}_d)$
where $\mathcal{O}_d$ is the ring of integers in $K_d$. Let $C_0 = C^{(S_0)}(G_0)$ and $C_d = C^{(S_d)}(G_d)$ be the corresponding congruence kernels. Then the natural
embedding $\Gamma_0 \to \Gamma_d$ induces a continuous homomorphism $\iota_d \colon C_0 \to C_d$. It follows from the results of \cite{ALR} that $\iota_d$ is injective for all $d$. (Indeed, by \cite[Theorem 8.1]{ALR}, the homomorphism of the profinite completions $\widehat{\mathrm{PSL}_2(\Z)} \to \widehat{\mathrm{PSL}_2(\mathcal{O}_d)}$ is injective, which
implies that the homomorphism $\widehat{\Gamma_0} \to \widehat{\Gamma_d}$ is injective, and the injectivity of $\iota_d$ follows. On the other hand, the results of Serre \cite{Serre1} imply that for $d \neq 1, 3$, the homomorphism $\iota_d$ is \emph{not} surjective. Moreover, we have the following.
\begin{lemma}\label{L:NSurj}
Let $E_d$ be the closed normal subgroup of \,$\widehat{\Gamma_d}$ generated by $\iota_d(C_0)$. Then for $d \neq 1, 3$, the quotient $C_d/E_d$ is infinite.
\end{lemma}
\begin{proof}
Since the image of $E_d$ in $\overline{C_d} := C_d / (C_d \cap \widehat{[\Gamma_d , \Gamma_d]})$ is the same as that of $C_0$, it is enough to show that the latter
has infinite index. It is well-known that the abelianization $\Gamma_0^{\small \mathrm{ab}} = \Gamma_0 / [\Gamma_0 , \Gamma_0]$ is finite (of order 12), so $C_0 \cap \widehat{[\Gamma_0 , \Gamma_0]}$ has finite index in $C_0$, making the image of $C_0$ in $\overline{C_d}$ finite. On the other hand, according to the results in \cite[\S 3.6]{Serre1}, for $d \neq 1, 3$, the abelianization $\Gamma_d^{\small \mathrm{ab}}$ is infinite\footnote{We note that for $d = 1, 3$, the abelianization $\Gamma_d^{\small
\mathrm{ab}}$ is finite as one can see from the explicit presentations found in \cite{Co} and \cite{Sw}.}.
Then from the exact sequence  $$\overline{C_d} \to \widehat{\Gamma_d}/\widehat{[\Gamma_d , \Gamma_d]} \to \overline{\Gamma_d} /\overline{[\Gamma_d , \Gamma_d]}$$ and the
finiteness of the last term in it (see \cite{Ra1}), we conclude that $\overline{C_d}$ is infinite, and our assertion follows.
\end{proof}

Nevertheless, we have the following in all cases.
\begin{prop}\label{P:Bian}
Let $L_d$ be the closed normal subgroup of $\widehat{G_d}$ generated by $\iota_d(C_0)$. Then $C_d/L_d$ is a finite cyclic group of order
dividing $\vert \mu_{K_d} \vert$ (so, its order is $\leqslant 2$ if $d \neq 1 , 3$, divides $4$ if $d = 1$, and $6$ if $d = 3$.)
\end{prop}
\begin{proof}
Pick a prime $p_0$ that does not split in $K_d$, and let $v_0$ the unique valuation of $K_d$ extending the $p_0$-adic valuation
of $\Q$. Consider the elements from Proposition \ref{P:SL2-gen} written for these valuations:
$$
c(p_0) = [\widehat{u}^+(1_{p_0}) \, , \, \widehat{u}^-(1'_{p_0})] \in C(G_0) \ \ \text{and} \ \ c(v_0) = [\widehat{u}^+(1_{v_0}) \, \ ,
\widehat{u}^-(1'_{v_0})] \in C(G_d).
$$
It is easy to see that these elements are related by $\iota_d(c(p_0)) = c(v_0)$. So, $L_d$ contains the subgroup $D$ from the statement of
Proposition \ref{P:SL2-gen}, and our claim follows from that proposition (cf.\,also Remark \ref{R:order}).
\end{proof}

\medskip

The proof of Theorem D in the general case will be reduced to the $\mathrm{SL}_2$-case
by constructing a suitable $K$-homomorphism $\mathrm{SL}_2 \to G$ with the help of
Jacobson-Morozov Lemma and then applying Proposition \ref{P:SL2-gen} in conjunction with the
following statement.
\begin{prop}\label{P:reduction}
Let $G$ be an absolutely simple simply connected algebraic $K$-group
of $K$-rank one, let $\varphi \colon H \to G$ be a $K$-homomorphism of an absolutely simple simply
connected  $K$-group $H$ to $G$,
and let $\widehat{\varphi} \colon \widehat{H}^{(S)} \to \widehat{G}^{(S)}$
be the corresponding continuous homomorphism of $S$-arithmetic completions.
Assume that $\varphi(H) \cap (U_{\beta} \setminus U_{2\beta}) \neq \varnothing$
for $\beta = \alpha$ and $-\alpha$. Let $C_0$ be a subgroup of  $C^{(S)}(H)$
normalized by $\widehat{H}^{(S)}$ and such that $\widehat{H}^{(S)}$ acts on $C^{(S)}(G)/C_0$ trivially. Then
for the closed normal subgroup $D$ of $\widehat{G}^{(S)}$ generated by $\widehat{\varphi}(C_0)$, the group
$\widehat{G}^{(S)}$ acts on $C^{(S)}(G)/D$ trivially. Consequently, $C^{(S)}(G)/D$ is a quotient of the metaplectic
kernel $M(S , G)$, hence it is a finite cyclic group of order dividing $\vert \mu_K \vert$.
\end{prop}

The proof requires one technical fact (Proposition \ref{P:Irred} below) which we will prove in the Appendix. To state it,
we observe that the centralizer $M = Z_G(T_s)$ of a maximal $K$-split torus $T_s$ of $G$ acts on each root subgroup $U_{\beta}$
for $\beta \in \Phi(G , T_s)$ via the adjoint action, and consequently acts on the quotient $W_{\pm \alpha} := U_{\pm \alpha}/U_{\pm 2\alpha}$.
Furthermore, it is known  $W_{\pm \alpha}$ is a vector group over $K$, and the above action gives rise to a $K$-linear representation
$\rho_{\pm \alpha} \colon M \to \mathrm{GL}(W_{\pm \alpha})$ - cf.\,\cite[\S21]{Bo}. We also recall that since $G$ has $K$-rank 1, its Tits index can have only
one or two circled vertices (cf.\,\cite{T}).
\begin{prop}\label{P:Irred}
Let $(W , \rho)$ denote either  $(W_{\alpha} , \rho_{\alpha})$ or
$(W_{-\alpha} , \rho_{-\alpha})$, and  assume that ${\rm char}\: K \neq 2.$ Then
$\rho$ is $K$-irreducible. More precisely, one of the
following two possibilities holds:

\vskip1mm

$(i)$ \parbox[t]{15cm}{the Tits index of $G$ has only one
circled node and then $\rho$ is absolutely irreducible;}

\vskip1mm

$(ii)$ \parbox[t]{15cm}{the Tits index of $G$ has two circled
nodes; then $W = W_1 \oplus W_2$ where $W_1$ and $W_2$ are
absolutely irreducible $M$-invariant subspaces defined over a
quadratic extension $L/K$ and $W_2 = W_1^{\sigma}$ for the
nontrivial automorphism $\sigma$ of $L/K.$ 
}
%
%
\end{prop}

\bigskip

\noindent {\it Proof of Proposition \ref{P:reduction}.} We only need to prove that the extension
\begin{equation}\label{E:G5}
1 \to  F := C^{(S)}(G)/D \longrightarrow \widecheck{G} := \widehat{G}^{(S)}/D  \stackrel{\theta}{\longrightarrow}  \overline{G}^{(S)} \to 1
\end{equation}
is central, for which we will use our standard strategy. More precisely, we let $\widecheck{U}_{\alpha}$ and $\widecheck{U}_{-\alpha}$ denote the closures
in $\widecheck{G}$ of $U_{\alpha}(K)$ and $U_{-\alpha}(K)$, respectively. Since the $S$-arithmetic and $S$-congruence topologies on $U_{\pm \alpha}(K)$ coincide, $\theta$ induces
isomorphisms
$$
\widecheck{U}_{\alpha} \longrightarrow \overline{U_{\alpha}(K)} \simeq U_{\alpha}(\mathbb{A}(S)) \ \ \text{and} \ \ \widecheck{U}_{-\alpha} \longrightarrow \overline{U_{-\alpha}(K)} \simeq
U_{-\alpha}(\mathbb{A}(S)),
$$
and we let $\sigma_{\pm \alpha} \colon U_{\pm \alpha}(\mathbb{A}(S)) \to \widecheck{U}_{\pm \alpha}$ denote the inverse (continuous) isomorphisms.
For $v \in V^K \setminus S$, we let
$\mathcal{G}_v$ denote the subgroup of $\widecheck{G}$ generated by $\sigma_{\alpha}(U_{\alpha}(K_v))$ and $\sigma_{-\alpha}(U_{-\alpha}(K))$. Repeating almost verbatim the argument used at the beginning of this section, we see that the subgroups $\mathcal{G}_v$ satisfy all the assumptions of Theorem B, which then yields the centrality of (\ref{E:G5}) provided we show that $\sigma_{\alpha}(U_{\alpha}(K_{v_1}))$ and $\sigma_{-\alpha}(U_{-\alpha}(K_{v_2}))$  commute elementwise for any $v_1 , v_2 \in V^K \setminus S$, $v_1 \neq v_2$. So, the central part of the present argument is concerned with proving this fact. We will establish it in the following equivalent form. Define
$$
c_{v_1 , v_2} \colon U_{\alpha}(K_{v_1}) \times U_{-\alpha}(K_{v_2}) \longrightarrow F, \ \ (u_1 , u_2) \mapsto [\sigma_{\alpha}(u_1) , \sigma_{-\alpha}(u_2)].
$$
Clearly, $c_{v_1 , v_2}$ is continuous, and what we need to prove is

\vskip2mm

\noindent $(\star)$ \  $c_{v_1 , v_2} \equiv 1$.

\vskip2mm

\noindent By our assumption, the extension
$$
1 \to F_0 := C^{(S)}(H)/C_0 \longrightarrow \widecheck{H} := \widehat{H}^{(S)}/C_0 \stackrel{\theta_0}{\longrightarrow} \overline{H}^{(S)} \to 1
$$
is central. Since $H$ is clearly $K$-isotropic, the congruence completion $\overline{H}^{(S)}$ can, as usual, be identified with $H(\mathbb{A}(S))$.
Then for any $v_1 , v_2 \in V^K \setminus S$, $v_1 \neq v_2$, by Corollary \ref{C:A1}, we can define a bimultiplicative pairing
$$
c_{v_1 , v_2}^0 \colon H(K_{v_1}) \times H(K_{v_2}) \longrightarrow F_0, \ \ (x_1 , x_2) \mapsto [\widetilde{x}_1 , \widetilde{x}_2] \ \ \text{for} \ \ \widetilde{x}_i \in
\theta_0^{-1}(x_i).
$$
As we already mentioned, the group $H(K_{v_i})$ does not contain any proper noncentral normal subgroups, hence $H(K_{v_i}) = [H(K_{v_i}) , H(K_{v_i})]$. Since $F_0$
is commutative, it follows that the pairing $c_{v_1 , v_2}^0$ is trivial, and therefore the pre-images $\theta_0^{-1}(H(K_{v_1}))$ and $\theta_0^{-1}(H(K_{v_2}))$
commute elementwise.

There exists unipotent $K$-subgroups $\mathscr{U}_+$ and $\mathscr{U}_-$ of $H$ such that $\varphi(\mathscr{U}_{\pm})$ is contained in $U_{\pm \alpha}$ but not
not in $U_{\pm 2\alpha}$. Then for any $u_1 \in \mathscr{U}_+(K_{v_1})$, $u_2 \in \mathscr{U}_-(K_{v_2})$ we have
$$
c_{v_1 , v_2}(\varphi(u_1) , \varphi(u_2)) = \widecheck{\varphi}(c_{v_1 , v_2}^0(u_1 , u_2)) = 1,
$$
where $\widecheck{\varphi} \colon \widecheck{H} \to \widecheck{G}$ is induced by $\widehat{\varphi}$.
Furthermore, the group $M(K)$ naturally acts on $F$ by conjugation, and for any $m \in M(K)$ and any $u_1 , u_2$ as above we have
\begin{equation}\label{E:G6}
c_{v_1 , v_2}(m \varphi(u_1) m^{-1} , m \varphi(u_2) m^{-1}) = m c_{v_1 , v_2}(\varphi(u_1) , \varphi(u_2)) m^{-1} = 1.
\end{equation}
We now note the following.
\begin{lemma}\label{L:M-WA}
{\rm (Weak approximation for $M$)} For any finite subset $V$ of $V^K$, the diagonal embedding $M(K) \hookrightarrow M_V := \prod_{v \in V} M(K_v)$ has
dense image.
\end{lemma}
\begin{proof}
By the Bruhat decomposition, the product map $\mu \colon U_{-\alpha} \times M \times U_{\alpha} \to G$ yields a $K$-isomorphism onto a Zariski-open set $\Omega
\subset G$. Being simply connected, $G$ has weak approximation with respect to any finite set of places, i.e. the diagonal embedding $G(K) \hookrightarrow G_V$ is dense
(cf.\,\cite[Theorem 7.8]{PlRa}). Since $\Omega$ is $K$-open, the diagonal embedding $\Omega(K) \hookrightarrow \Omega_V$ is also dense, and our assertion follows.
\end{proof}

Using this in conjunction with (\ref{E:G6}) and the continuity of $c_{v_1 , v_2}$, we obtain that
$$
c_{v_1 , v_2}(X_1(u_1) , X_2(u_2)) = \{ 1 \} \ \ \text{where} \ \ X_i(u_i) = \{ m_i \varphi(u_i) m_i^{-1} \ \vert \ m_i \in M(K_{v_i}) \}.
$$
Then also
\begin{equation}\label{E:G8}
c_{v_1 , v_2}(\langle X_1(u_1) \rangle , \langle X_2(u_2) \rangle) = \{ 1 \},
\end{equation}
where $\langle X_i(u_i) \rangle$ is the subgroup generated by $X_i(u_i)$. Now, it follows from our assumptions and the Zariski-density of
$\mathscr{U}_{\pm}(K)$ in $\mathscr{U}_{\pm}$ that one can pick $u_1 \in \mathscr{U}_+(K)$ and $u_2 \in \mathscr{U}_-(K)$ so that $\varphi(u_1)
\notin U_{2\alpha}(K)$ and $\varphi(u_2) \notin U_{-2\alpha}(K)$. So, if we let $\nu_{\pm \alpha} \colon U_{\pm \alpha} \to
U_{\pm \alpha}/U_{\pm 2\alpha} =: W_{\pm \alpha}$ denote the quotient map, then $w_1 = \nu_{\alpha}(\varphi(u_1))$ and
$w_2 = \nu_{-\alpha}(\varphi(u_2))$ are {\it nontrivial} elements in $W_{\alpha}(K)$ and $W_{-\alpha}(K)$. Taking into account the Zariski-density
of $M(K)$ in $M$ (cf. \cite[18.3]{Bo}) and applying Proposition \ref{P:Irred}, we see that for any field extension $P/K$, the $P$-vector space $W_{\alpha}(P)$ (resp.,
$W_{-\alpha}(P)$) is spanned by $\rho_{\alpha}(M(P)) \cdot w_1$ (resp., $\rho_{-\alpha}(M(P)) \cdot w_2$). On the other hand, since $\alpha(T_s(P))$ contains
${P^{\times}}^d$ for some $d \geqslant 1$ and $P$ is generated by ${P^{\times}}^d$ as an additive group, the additive subgroup of \,$W_{\alpha}(P)$ (resp., $W_{-\alpha}(P)$)
generated by   $\rho_{\alpha}(M(P)) \cdot w_1$ (resp., $\rho_{-\alpha}(M(P)) \cdot w_2$) is automatically a $P$-vector subspace. Altogether, this means that
\begin{equation}\label{E:G9}
\nu_{\alpha}(\langle X_1(u_1) \rangle) = W_{\alpha}(K_{v_1}) \ \ \text{and} \ \ \nu_{-\alpha}(\langle X_2(u_2) \rangle) = W_{-\alpha}(K_{v_2}).
\end{equation}
Clearly, $U_{\pm 2\alpha}$ is contained in the center of $U_{\pm \alpha}$, and since $U_{\pm 2 \alpha}(P)$ coincides with the commutator subgroup of $U_{\pm \alpha}(P)$
for any field extension $P/K$ (cf.\,\cite[5.3]{BuH} - note that this fact is true over any infinite field of characteristic $\neq 2$), we obtain from (\ref{E:G9}) by
passing to commutator subgroups that $\langle X_1(u_1) \rangle$ (resp., $\langle X_2(u_2) \rangle$) contains $U_{2\alpha}(K_{v_1})$ (resp., $U_{-2\alpha}(K_{v_2})$). Then
(\ref{E:G9}) yields
$$
X_1(u_1) = U_{\alpha}(K_{v_1}) \ \ \text{and} \ \ X_2(u_2) = U_{-\alpha}(K_{v_2}).
$$
Combining this with (\ref{E:G8}), we obtain $(\star)$, as required. \hfill $\Box$

\medskip

\begin{remark}
(1) Proposition \ref{P:reduction} for $C = C_0$ is essentially due to Rajan and Venkataramana \cite{RjV} and in fact
goes back to Raghunathan's argument in \cite[\S 3]{Ra2}. We note, however, that the discussion of the irreducibility
of the action of $M$ on $W_{\pm \alpha}$ (which is our Proposition \ref{P:Irred}) is limited in \cite{RjV} to the groups
$\mathrm{SO}(n , 1)$ and $\mathrm{SU}(n , 1)$ which are the main focus of that paper - see the paragraph before last on p.\,548.
It should also be pointed out that the assertion in the proof of Theorem 7 in \cite{RjV} that part $(ii)$ of that theorem is
a restatement of \cite[Proposition 2.14]{Ra2} is not totally accurate as Proposition 2.14 of \cite{Ra2} involves one extra condition -
see $(iii)$ in its statement. Nevertheless, according to our Theorem B, the result described in \cite[Theorem 7(ii)]{RjV} is indeed valid, and not only for isotropic groups. In view of these technicalities, we chose - for the reader's convenience - to give a complete proof of Proposition
\ref{P:reduction}.

(2) It was pointed out in \cite{RaV} and \cite{RjV} that the assertion of Proposition \ref{P:reduction} has the following implication:

{\it Given a congruence subgroup $\Gamma$ of $G(\mathcal{O}(S))$ and a nontrivial group homomorphism $\phi \colon \Gamma \to \Z$, there exists a congruence subgroup
$\Delta$ of $H(\mathcal{O}(S))$ and an element $g \in G(K)$ such that $\Delta' = g\Delta g^{-1}$ is contained in $\Gamma$ and the restriction $\phi \vert \Delta'$
is nontrivial.}
\vskip1mm

\noindent This is subsumed, however, in the ``Sandwich Lemma'' of Lubotzky \cite[Lemma 2.4]{Lu2}, which states that the above result is valid
without any assumptions on the congruence kernels {\it if} $H(\mathcal{O}(S))$ satisfies the so-called {\it Selberg property}. We refer to \cite{Lu2}
for precise definitions, and only mention that the Selberg property is in fact property $(\tau)$ for congruence subgroups. More importantly,
 the Selberg property is now known to hold in all situations (see \cite{Cl}, which concluded the efforts by various people), making the result of
Lubotzky unconditional.

At the same time, proving Selberg's property even for $\mathrm{SL}_2$ requires the heavy machinery of the theory of automorphic forms,
so the approach developed in \cite{RaV} and \cite{RjV} provides an algebraic alternative in some cases. (From this perspective, our Proposition
\ref{P:Bian} yields an algebraic proof of the following fact: {\it Given a congruence subgroup $\Gamma$ of the Bianchi group $\mathrm{SL}_2(\mathcal{O}_d)$,
where $\mathcal{O}_d$ is the ring of integers in $K_d = \Q(\sqrt{-d})$ with $d$ a square-free integer $> 0$, and a nontrivial homomorphism $\phi \colon
\Gamma \to \Z$, there exists a congruence subgroup $\Delta$ of $\mathrm{SL}_2(\Z)$ and $g \in \mathrm{SL}_2(K_d)$ such that $\Delta' = g \Delta g^{-1}$ is
contained in $\Gamma$ and the restriction $\phi \vert \Delta'$ is nontrivial} (this should be compared to the results in \cite[3.6]{Serre1})).

(3) One can ask whether it is possible to strengthen Proposition \ref{P:reduction} and prove that for an absolutely almost simple simply connected
$K$-group $G$ and a proper  $K$-subgroup $H$, the map of the congruence kernels $\iota^{(S)}_{G , H} \colon C^{(S)}(H) \to C^{(S)}(G)$ is actually surjective.
This property can be helpful for proving the centrality of $C^{(S)}(G)$ in view of the following simple observation (cf.\,Proposition 2 in \cite[5.2]{PR-Milnor}):
{\it Assume that $G(K)$ does not contain any proper noncentral normal subgroups. If there exists a $K$-subgroup $H$ of $G$ which is fixed elementwise by
a nontrivial $K$-automorphism $\sigma$ of $G$ such that $\iota^{(S)}_{G , H}$ is surjective, then $C^{(S)}(G)$ is central, hence finite.}
This observation (which can be traced back to
\cite{BMS} - see \cite[5.3]{PR-Milnor} on how it can be used to establish the centrality of the congruence kernel for $\mathrm{SL}_n$, $n \geqslant 3$) was employed
by Kneser \cite{Kn} to prove that if $G = \mathrm{Spin}_n(q)$ is the spinor group of a nondegenerate quadratic form $q$ over $K$ in $n \geqslant 5$ variables and
$\mathrm{rk}_S\: G \geqslant 2$, then $C^{(S)}(G)$ is central. To this end, he proved that for any anisotropic $x \in K^n$ with the stabilizer $G(x)$ satisfying
$\mathrm{rk}_S\: G(x) \geqslant 1$, the map $C^{(S)}(G(x)) \to C^{(S)}(G)$ is surjective (see Proposition 3 in \cite[5.2]{PR-Milnor} for an indication of the idea).
Subsequently, analogues of these statements were established for groups of the classical types and type $\textsf{G}_2$ in \cite{Ra-G2}, \cite{Ra-CSP}, \cite{Ra-Hab},
\cite{To1}, \cite{To2}. To give an example where $\iota^{(S)}_{G , H}$ is not surjective, we consider an imaginary quadratic extension $L/\Q$ and let $h$ be the corresponding
2-dimensional hyperbolic hermitian form. Set $f = h \perp g$, where $g$ is a 1-dimensional hermitian form, and consider the natural embedding of (absolutely almost simple,
simply connected) $\Q$-groups
$$
H := \mathrm{SU}_2(h) \to \mathrm{SU}_3(f) =: G.
$$
We claim that for $S = \{ \infty \}$, the map $\iota^{(S)}_{G , H}$ is \emph{not} surjective. Indeed, it follows from the results of Kazhdan \cite{Kaz} and
Wallach \cite{Wa} that there exists a congruence subgroup $\Gamma$ of $G(\Z)$ with infinite abelianization $\Gamma^{\small \mathrm{ab}}$, which immediately implies
that the congruence kernel $C^{(S)}(G)$ is infinite (cf. \cite[\S 3]{Serre1}). Since $H$ is fixed by the nontrivial automorphism $\sigma = \mathrm{Int} \: x$ of $G$, where
$x = \mathrm{diag}(1, 1, -1) \in \mathrm{U}_3(f)$, this fact together with the above observation prevents $\iota^{(S)}_{G , H}$ from being surjective.

While $\iota^{(S)}_{G , H}$ may or may not be surjective, the available results (including those obtained in \cite{RaV} for the embeddings $\mathrm{SO}(2m-1 , 1)
\to \mathrm{SU}(2m-1 , 1)$ and $\mathrm{SO}(2m-1 , 1) \to \mathrm{SO}(2m+1 , 1)$ in the anisotropic case and $C = C_0$) suggest that the assertion of Proposition
\ref{P:reduction} should always be true whenever $G$ and $H$ are absolutely almost simple simply connected $K$-groups and $\mathrm{rk}_S\: H > 0$. If proven, this would
simplify the verification of centrality in a number of cases.
\end{remark}

\medskip

{\it Proof of Theorem D.} Recall that here $\mathrm{char}\: K = 0$. It follows from Propositions
\ref{P:SL2-gen} and \ref{P:reduction} that it is enough to construct a
$K$-homomorphism $\varphi \colon H = \mathrm{SL}_2 \to G$
such that $\varphi(H) \cap (U_{\beta} \setminus U_{2\beta}) \neq \varnothing$
for $\beta = \alpha$ and $-\alpha$. For this we consider the Lie algebra
$\mathfrak{g} = L(G)$ of $G$, and pick a nonzero eigenvector $X \in \mathfrak{g}(K)$ for
the adjoint action of the maximal $K$-split torus $T_s$ with character $\alpha$. Applying
the Jacobson-Morozov Lemma (cf.\,\cite{J}, Ch.\,III, Thm.\,17), we can find a
$K$-subalgebra $\mathfrak{r} \subset \mathfrak{g}$ that contains $X$ and is isomorphic
to $\mathfrak{sl}_2$. There exists an algebraic $K$-subgroup $R$ of $G$ with
the Lie algebra $\mathfrak{r}$ (cf. \cite{Bo}, Cor. 7.9),
which is $K$-isogenous to $\mathrm{SL}_2$. Let $\mathscr{U}$ be a 1-dimensional
unipotent $K$-subgroup of $R$ whose Lie algebra $L(U)$ is spanned by $X$, and let $\mathscr{T}$ be
a 1-dimensional $K$-split torus that normalizes $\mathscr{U}$.  Then $\mathscr{T}$ and and $T_s$ are
conjugate by an element of $N_G(\mathscr{U})^{\circ}(K)$ (cf. \cite{BT}), and after performing
this conjugation we can assume that $\mathscr{T} = T_s$. Since $\mathfrak{r}$ also contains an eigenvector
for $\mathrm{Ad}\: \mathscr{T}$ with character $-\alpha$, we obtain that $R \cap (U_{\beta} \setminus U_{2\beta})
\neq \varnothing$ for $\beta = \pm \alpha$, so a $K$-isogeny $\varphi \colon H = \mathrm{SL}_2 \to R$ is a required
homomorphism.  \hfill $\Box$

\medskip

To conclude, we will briefly indicate how Theorem D can be partially extended to positive
characteristic $p > 2$. The main distinction is that if $p > 0$ and $\mathrm{rk}_S\: G = 1$ then the arithmetic and congruence
topologies of $G$ may not coincide on $U_{\pm \alpha}(K)$.  So, to
use our approach we need to pass to the {\it reduced} congruence
kernel $\overline{C}^{(S)}(G) = C^{(S)}(G)/N$ where $N$ is the closed normal subgroup of
$\widehat{G}^{(S)}$ generated by the kernels of the restrictions $\pi^{(S)}
\vert \widehat{U_{\pm \alpha}(K)}$. Then Propositions \ref{P:SL2-gen} and
\ref{P:reduction} remain valid if one replaces the full congruence kernel
with the reduced one in their statements. Furthermore, by going
through the list of absolutely almost simple groups defined over a global
field $K$ of characteristic $p > 2$ and having $K$-rank one, one
verifies that there is a $K$-homomorphism $\varphi \colon H
= \mathrm{SL}_2 \to G$ such that $\varphi(H) \cap (U_{\beta} \setminus
U_{2\beta}) \neq \varnothing$ for $\beta = \pm \alpha$(this fact being false in
characteristic two). This puts all the ingredients of the proof of
Theorem D in place, and taking into account the computation of $M(S , G)$ in positive
characteristic (cf. \cite{PR1}), we arrive at the following conclusion:
{\it $\overline{C}^{(S)}(G)$ is generated as a closed normal subgroup of $\widehat{G}^{(S)}/N$
by a single element.}

\newpage

\centerline{\sc Appendix: Proof of Proposition \ref{P:Irred}}

\vskip5mm

We will give the argument for $(W_{\alpha} , \rho_{\alpha})$. Let
$T$ be a maximal $K$-torus of $G$ containing $T_s$, and let
$\Phi = \Phi(G , T)$ be the corresponding (absolute) root system. We
fix compatible orderings on $X(T) \otimes_{\Z} \R$ and $T(T_s) \otimes_{\Z} \R$
so that $\alpha$ is positive. Let $\Phi^+$ (resp., $\Delta$) be the corresponding system
of positive (resp., simple) roots in $\Phi$. Furthermore, we let $\Delta_0$ denote the
subset of $\Delta$ consisting of roots with trivial restriction to $T_s$ (and then $\Delta \setminus \Delta_0$ is the
set of distinguished roots). Since $\mathrm{rk}_K \: G = 1$, it follows from the tables in \cite{T} that
$\mid\Delta \setminus \Delta_0\mid \leqslant 2$; note that any $\delta \in \Delta \setminus \Delta_0$ is taken to $\alpha$
by the restriction map  $X(T) \to X(T_s)$.

For $\beta \in \Phi$, we let $\mathscr{U}_{\beta}$
(resp., ${\mathfrak g}_{\beta}$) denote the 1-dimensional connected
unipotent subgroup of $G$ (resp., the 1-dimensional subspace of the
Lie algebra ${\mathfrak g} = L(G)$) corresponding to $\beta$ (thus, ${\mathfrak g}_{\beta}
= L({\mathscr{U}}_{\beta})$). Furthermore, we let $n_{\delta}(\beta)$ $(\delta \in \Delta)$
denote the integers that arise in the  decomposition $\beta = \sum_{\delta
\in \Delta} n_{\delta}(\beta)\delta$. Let $$\Theta =
\{ \beta \in \Phi^+ \mid \sum_{\delta \in \Delta - \Delta_0}
n_{\delta}(\beta) = 1 \}.$$
Clearly, $\Theta$  is precisely the set of roots $\beta \in \Phi$ that restrict
to $\alpha$. It follows that $\mathfrak{u} = \sum_{\beta \in \Theta} \mathfrak{g}_{\beta}$ is the eigenspace for $T_s$ for
the character $\alpha$, hence is invariant under $\mathrm{Ad}\: M$, where $M = Z_{G}(T_s)$. It is well-known that the vector spaces
$W_{\alpha} = U_{\alpha}/U_{2\alpha}$ and $\mathfrak{u}$ are isomorphic as $M$-modules.
We also recall that for $\beta , \gamma \in \Phi$, we have
\begin{equation}\tag{A.1}\label{E:A1}
(\mathrm{Ad}\: g)({\mathfrak g}_{\beta}) \subset \sum_{n \geqslant 1} {\mathfrak g}_{\beta + n\gamma}  \ \ \text{for any} \ \ g \in \mathscr{U}_{\gamma},
\end{equation}
where as usual we set ${\mathfrak g}_{\delta} = 0$ if $\delta \in X(T)$ is not a root. Furthermore, since we exclude characteristic 2 and also
type $\textsf{G}_2$ (which does not have $K$-forms with $K$-rank 1), we have
\begin{equation}\tag{A.2}\label{E:A2}
[\mathfrak{g}_{\beta_1} , \mathfrak{g}_{\beta_2}] = \mathfrak{g}_{\beta_1 + \beta_2} \ \ \text{for any} \ \ \beta_1 , \beta_2 \in \Phi.
\end{equation}

\vskip2mm

\noindent {\bf Lemma A.1.} {\em
Fix $\delta_0 \in \Delta \setminus \Delta_0$, and set $\Theta(\delta_0) = \{ \beta
\in \Theta \mid n_{\delta_0}(\beta) = 1 \}$. Then $${\mathfrak u}(\delta_0)
:= \sum_{\beta \in \Theta(\delta_0)} {\mathfrak g}_{\beta}$$ is an
irreducible $M$-module.}

\vskip2mm

\begin{proof}
The group $M$ is generated by $T$ and $\mathscr{U}_{\gamma}$ for those $\gamma \in \Phi$ that restrict trivially to $T_s$.
Since any such $\gamma$ is a linear combination of elements of $\Delta_0$, the inclusion (\ref{E:A1}) shows that
${\mathfrak u}(\delta_0)$ is $\mathrm{Ad}\: M$-invariant. Let ${\mathfrak v} \subset {\mathfrak u}(\delta_0)$ be
a nonzero $M$-invariant subspace. As $M$ contains $T$, we have ${\mathfrak v} = \bigoplus_{\beta \in \Theta'} {\mathfrak
g}_{\beta}$ for some nonempty subset $\Theta' \subset \Theta(\delta_0)$
and $[{\mathfrak m} , {\mathfrak v}] \subset {\mathfrak v}$, where
${\mathfrak m} = L(M)$. For $\beta_1 , \beta_2 \in
\Phi$, we will write $\beta_1 \succ \beta_2$ if $\beta_1 - \beta_2$ is
a sum of positive roots. We claim that if $\beta_1 , \beta_2 \in
\Theta(\delta_0)$ and $\beta_1 \succ \beta_2$, then
\begin{equation}\tag{A.3}\label{E:A3}
{\mathfrak g}_{\beta_1} \subset {\mathfrak v} \ \ \
\Leftrightarrow \ \ \ {\mathfrak g}_{\beta_2} \subset {\mathfrak
v}
\end{equation}
Indeed, according to \cite[Ch.\,VI, \S 1, n$^{\circ}$ 6, Prop.\,19]{Bour}, there
exists a sequence of positive roots $\gamma_1, \ldots , \gamma_r \in \Phi^+$
such that $\beta_1 = \beta_2 + \gamma_1 + \cdots + \gamma_r$ and
$\beta_2 + \gamma_1 + \cdots + \gamma_i$ is a root for  $i = 1,
\ldots , r$. Since $n_{\delta}(\beta_1) = n_{\delta}(\beta_2)$ for any $\delta
\in \Delta \setminus \Delta_0$, we have $n_{\delta}(\gamma_i) = 0$, hence $\mathfrak{g}_{\pm \gamma_i}
\subset \mathfrak{m}$, for all $i$. So, if $\mathfrak{g}_{\beta_1} \subset \mathfrak{v}$, then using repeatedly $[\mathfrak{m}
, \mathfrak{v}] \subset \mathfrak{v}$  together with (\ref{E:A2}), we obtain
$$
{\mathfrak g}_{\beta_2} = [{\mathfrak g}_{-\gamma_1} , [{\mathfrak
g}_{-\gamma_2} , [ \cdots [{\mathfrak g}_{-\gamma_r} , {\mathfrak
g}_{\beta_1}] \cdots ] \subset {\mathfrak v},
$$
and vice versa, proving (\ref{E:A3}). Note that for any $\beta \in \Theta(\delta_0)$ we have $\beta \succ \delta_0$,
so using (\ref{E:A3}), we see that if $\mathfrak{g}_{\beta_0} \subset \mathfrak{v}$ for \emph{some} $\beta_0 \in \Theta(\delta_0)$, then
$\mathfrak{g}_{\delta_0} \subset \mathfrak{v}$, and consequently  $\mathfrak{g}_{\beta} \subset \mathfrak{v}$ for \emph{every} $\beta
\in \Theta(\delta_0)$. Thus, $\mathfrak{v} = \mathfrak{u}$, as claimed.
\end{proof}

If $\Delta \setminus \Delta_0 = \{ \delta_0 \}$ then the above lemma, together
with the remarks made prior to its statement,
immediately yields the irreducibility of $W_{\alpha}.$ Now, suppose that $\Delta
\setminus \Delta_0 = \{ \delta_1 , \delta_2 \}$. For $i = 1, 2$, set $$\mathfrak{u}_i
= \sum_{\beta \in \Theta(\delta_i)} \mathfrak{g}_{\beta},$$ where $\Theta(\delta_i)$ is the subset
of $\Theta$ defined in Lemma A.1 for $\delta_0 = \delta_i$, and let $W_i$ be the subspace of $W$ corresponding to
$\mathfrak{u}_i$. Then clearly $W = W_1 \bigoplus W_2$, and according to Lemma A.1, each $W_i$ is an (absolutely) irreducible
$M$-module. Let $T_0 = Z(M)^{\circ}$ be the central torus of the (reductive) group $M$. Then the restrictions
$\gamma_i = \delta_i \vert T_0$ for $i = 1, 2$ form a basis of $X(T_0) \otimes_{\Z} \Q$, and $W_i$ is the eigenspace of $T_0$ with
the character $\gamma_i$. It follows that the $M$-submodule of $W$ containing $w = (w_1 , w_2)$ with $w_i \in W_i$, contains
$w_1$ and $w_2$, hence coincides with $W$ if both $w_1$ and $w_2$ are nonzero.

Since $T_0$ is 2-dimensional and contains the 1-dimensional (maximal) split torus $T_s$, it splits over a quadratic extension $L/K$, and then
both subspaces $W_1$ and $W_2$ are defined over $L$. The nontrivial $\sigma \in \mathrm{Gal}(L/K)$ can either switch the weights $\gamma_1 , \gamma_2$
of $T_0$, or keep each of them fixed. However, in the second option $T_0$ would be $K$-split, which is not the case. Thus, $\sigma(\gamma_1) = \gamma_2$, and
therefore $\sigma(W_1) = W_2$. It follows that if a nonzero $w \in W(K)$ is written in the form $w = (w_1 , w_2)$ as above then both $w_1$ and $w_2$ are
\emph{automatically} nonzero, so $w$ generates $W$ as $M$-module, implying that $W$ is $K$-irreducible.   \hfill $\Box$


\bigskip

{\small {\bf Acknowledgements.} Both authors were supported by the NSF (grants DMS-1401380   and DMS-1301800).} We thank the referee for her/his comments.

\vskip8mm

\bibliographystyle{amsplain}

\begin{thebibliography}{100}

\bibitem{ALR} I.~Agol, D.D.~Long, A.W.~Reid, {\it The Bianchi groups are separable
on geometrically finite subgroups,} Ann. of Math.  {\bf 153}(2001),
no. 3, 599--621.

\bibitem{ANT} {\it Algebraic Number Theory,} edited by J.W.S.~Cassels and A.~Fr\"ohlich, 2nd edition,
London Math. Soc., 2010.


\bibitem{BMS} H.~Bass, J.~Milnor, J-P.~Serre, {\it Solution of
the congruence subgroup problem for $\mathrm{SL}_n$ $(n \geqslant
3)$ and $\mathrm{Sp}_{2n}$ $(n \geqslant 2)$,} Publ. math. IHES {\bf
33}(1967), 54-137.


\bibitem{Ber-Shap} V.~Bergelson, D.B.~Shapiro, {\it Multiplicative subgroups of finite index in a ring}, Proc. AMS {\bf 116}(1992),
885-896.


\bibitem{Bo} A.~Borel, {\it Linear Algebraic Groups,} GTM 126,
Springer, 1991.

\bibitem{BT} A.~Borel, J.~Tits, {\it Groupes r\'eductifs,} Publ.
math. IHES {\bf 27}(1965), 55-150.


\bibitem{BoT-Hom} A.~Borel, J.~Tits, {\it Homomorphismes ``abstraits'' de groupes alg\'ebriques simples,} Ann. Math. {\bf 97}(1973), 499-571.


\bibitem{Bour} N.~Bourbaki, {\it Groupes et alg\`ebres de Lie,}
Hermann, Paris, Ch.\,III -- 1972, Ch.\,IV-VI -- 1968.

\bibitem{BuH} C.~Bushnell, G.~Henniart, {\it On the derived subgroups of
certain unipotent subgroups of reductive groups over infinite
fields,} Transform. Groups {\bf 7}(2002), no. 3, 211--230.


\bibitem{Cl} L.~Clozel, {\it D\'emonstration de la conjecture
$\tau,$} Invent. math. {\bf 151}(2003), 297-328.

\bibitem{Co} P.M.~Cohn, {\it A presentation of $\mathrm{SL}_2$
for Euclidean imaginary quadratic number fields,} Mathematika {\bf
15}(1968), 156-163.

\bibitem{Dem} C.~Demarche, {\it Le d\'efaut d'approximation forte dans les groupes
lin\'eaires connexes,} Proc. London Math. Soc. {\bf 102}(2011), 563-597.


\bibitem{D} J.~Dixon, {\it The Structure of Linear Groups,}
Van Nostrand, Princeton, NJ, 1971.

\bibitem{Anal} J.D.~Dixon, M.P.F.~du Sautoy, A.~Mann and D.~Segal, {\it Analytic pro-$p$ Groups}, 2nd edition,
Cambridge Univ. Press, 1999.


\bibitem{Gille} P.~Gille, {\it Le probl\`eme de Kneser-Tits}, S\'eminaire Bourbaki,
2007/2008, exp. 983. Ast\'erisque No. 326(2009), 39-81.

\bibitem{J} N.~Jacobson, {\it Lie algebras}, Dover Publications, New York, 1972.


\bibitem{JaK} A.\:Jaikin-Zapirain, B.\:Klopsch, {\it Analytic groups over general pro-$p$ domains,} J. London Math. Soc. {\bf 76}(2007), 365-383.

\bibitem{Kaz} D.~Kazhdan, {\it Some applications of the Weil representation}, J. Analyse Mat. {\bf 32}(1977),
235-248.

\bibitem{LPy} M.W.~Liebeck and L.~Pyber, {\it Finite linear groups and
bounded generation,} Duke Math. J. {\bf 107}(2001), No.~1, p. 159-171.

\bibitem{Kn} M.~Kneser, {\it Normalteiler ganzzahliger
Spingruppe,} J. Reine Angew. Math. {\bf 311-312}(1979), 191-214.

\bibitem{Lu0} A.~Lubotzky, {\it Free quotients and the congruence
kernel of $\mathrm{SL}_2,$} J. Algebra {\bf 77}(1982), 411-418.

\bibitem{Lu1} A.~Lubotzky, {\it Subgroup growth and congruence subgroups,}
Invent. math. {\bf 119}(1995), 267-295.

\bibitem{Lu2} A.~Lubotzky, {\it Eigenvalues of the Laplacian, the
First Betti Number and the Congruence Subgroup Problem,} Ann. math.
{\bf 144}(1996), 441-452.

\bibitem{Lu3} A.~Lubotzky, {\it Finite presentation of adelic
groups, the congruence kernel and cohomology of finite simple
groups,} Pure Appl. Math. Q. {\bf 1}(2005), no. 2, part 1, 241-256.

\bibitem{LuSe} A.~Lubotzky, D.~Segal, {\it Subgroup Growth,} Birkh\"auser, 2003.

\bibitem{Ma-SA} G.A.~Margulis, {\it Cobounded subgroups in
algebraic groups over local fields} (in Russian), Funkc. Anal. i
Prilo\u{z}. {\bf 11}(1977), No. 2, 45-57 (English translation:
Functional Anal. Appl. {\bf 11}(1977), No. 2, 119-122).

\bibitem{MSZ} A.W.~Mason,A.~Premet,  B.~Sury, P.A.~Zalesskii, {\it The congruence
kernel of an arithmetic lattice in a rank one algebraic group over a
local field,} J. Reine Angew. Math.  {\bf 623}(2008), 43-72.

\bibitem{Me} O.V.~Melnikov, {\it Congruence kernel for the group
$\mathrm{SL}_2(\Z),$} Dokl. Akad. Nauk SSSR {\bf 228}(1976),
1034-1036.

\bibitem{Pl-SA} V.P.~Platonov, {\it The problem of strong
approximation and the Kneser-Tits conjecture for algebraic groups}
(in Russian), Izv. Akad. Nauk SSSR. Ser. mat. {\bf 33}(1969), No. 6,
1211-1219 (English translation: Mathematics of the USSR-Izvestiya
{\bf 3}(1969), No. 3, 1139-1147).


\bibitem{PlRa} V.P.~Platonov and A.S.~Rapinchuk, {\it Algebraic Groups
and Number Theory}, Academic Press, Boston (1994).

\bibitem{PR2} V.P.~Platonov and A.S.~Rapinchuk, {\it Abstract properties
of $S$-arithmetic groups and the congruence subgroup problem,} Russian
Akad. Sci. Izv. Math. {\bf 40}(1993), 455-476.

\bibitem{PS} V.P.~Platonov and  B.~Sury, {\it Adelic profinite groups,}
J.~Algebra {\bf 193}(1997), p. 757-763.


\bibitem{Pr-SA} G.~Prasad, {\it Strong approximation for
semi-simple groups over function fields}, Ann. math. {\bf
105}(1977), 553-572.

\bibitem{Pr-RT} G.~Prasad, {\it Elementary proof of a theorem of Bruhat-Tits-Rousseau and of a theorem of Tits,}
Bull. Soc. Math. France {\bf 110}(1982), 197-202.

\bibitem{Pr-Rag} G.~Prasad and M.S.~Raghunathan, {\it On the Kneser-Tits problem}, Comment. Math. Helv. {\bf 60}(1885),
107-121.


\bibitem{PR1} G.~Prasad and A.S.~Rapinchuk, {\it Computation of the
metaplectic kernel,} Publ. Math. IHES\:{\bf 84}(1996), 91-187.


\bibitem{PR-Irred} G.~Prasad and A.S.~Rapinchuk, {\it Irreducible tori in semisimple groups},
Intern. Math. Res. Notices, 2001, No. 23, 1229-1242.

\bibitem{PR-Irred-Er} G.~Prasad and A.S.~Rapinchuk, {\it Irreducible tori in semisimple
groups (Erratum)}, Intern. Math. Res. Norices, 2002, No. 17, 919-921.

\bibitem{PR-Milnor} G.~Prasad and A.S.~Rapinchuk, {\it Developments on the congruence
subgroup problem after the work of Bass, Milnor and Serre}, in: Collected works of John Milnor, vol. V `Algebra,' AMS 2010,
307-325.

\bibitem{Ra1} M.S.~Raghunathan, {\it On the congruence subgroup
problem,} Publ. math. IHES {\bf 46}(1976), 107-161.

\bibitem{Ra2} M.S.~Raghunathan, {\it On the congruence subgroup
problem II,} Invent. math. {\bf 85}(1986), 73-117.

\bibitem{Rag-SL1} M.S.~Raghunathan, {\it On the group of norm 1 elements in a division
algebra}, Math. Ann. {\bf 279}(1988), 457-484.

\bibitem{Ra3} M.S.~Raghunathan, {\it The congruence subgroup
problem,} Proc. of the Hyderabad Conf. on Algebraic Groups
(Hyderabad, 1989), Manoj Prakashan, Madras, 1991, 465-494.


\bibitem{RaV} M.S.~Raghunathan, T.N.~Venkataramana, {\it The
first Betti number of arithmetic groups and the congruence subgroup
problem,} Contemp. math. {\bf 153}(1993), 95-107.


\bibitem{RjV} C.S.~Rajan, T.N.~Venkataramana, {\it On the first
cohomology of arthmetic groups,} Manuscripta math. {\bf 105}(2001),
537-552.

\bibitem{Ra-G2} A.S.~Rapinchuk, {\it The congruence subgroup problem for algebraic groups and
strong approximation for affine varieties}, Dokl. Akad. Nauk BSSR {\bf 32}(1988), 581-584 (in Russian).

\bibitem{Ra-CSP} A.S.~Rapinchuk, {\it On the congruence subgroup problem for algebraic groups},
Soviet Math. Dokl. {\bf 39}(1989), 618-621.


\bibitem{Ra-Comb} A.S.~Rapinchuk, {\it Combinatorial Theory of Arithmetic Groups}, Preprint 20(420), Institute
of Mathematics, Academy of Sciences of the BSSR, 1990.

\bibitem{Ra-Hab} A.S.~Rapinchuk, {\it The congruence subgroup problem}, Habilitationsschrift, Institute of Mathematics,
Academy of Sciences of the BSSR, 1990.

\bibitem{Ra-Bass} A.S.~Rapinchuk, {\it The congruence subgroup problem}, Algebra, K-theory, Groups and Education (on the occasion
of Hyman Bass's 65th birthday), ed. by T.Y.~Lam and A.R.~Magid, Contemp. math. {\bf 243}(1999), 173-188.

\bibitem{R-MP} A.S.~Rapinchuk, {\it The Margulis-Platonov conjecture for $\mathrm{SL}_{1 , D}$ and
2-generation of finite simple groups}, Math. Z. {\bf 252}(2006), 295-313.


\bibitem{Ra-SA} A.S.~Rapinchuk, {\it Strong approximation for
algebraic groups}, Thin Groups and Superstrong Approximation, MSRI
Publications, {\bf 61}(2014), 269-298.



\bibitem{RR} A.S.~Rapinchuk, I.A.~Rapinchuk, {\it Centrality of the congruence kernel for elementary
subgroups of Chevalley groups of rank $> 1$ over Noetherian rings}, Proc. AMS {\bf 139}(2011), 3099-3113.




\bibitem{RS} A.S.~Rapinchuk, Y.~Segev, {\it Valuation-like maps and the congruence
subgroup property}, Invent. math. {\bf 144}(2001), 571-607.


\bibitem{RSS} A.S.~Rapinchuk, Y.~Segev, G.M.~Seitz, {\it Finite quotients of the multiplicative
group of a finite dimensional division algebra are solvable}, JAMS {\bf 15}(2002), 929-978.

\bibitem{Ri} C.~Riehm, {\it The congruence subgroup problem over local fields,} Amer. J. Math. {\bf 92}(1970), 771-778.



\bibitem{Segev} Y.~Segev, {\it On finite homomorphic images of the multiplicative group of a division algebra},
Ann.\:Math.\, {\bf 149}(1999), 219-251.

\bibitem{Serre1} J-P.~Serre, {\it Le probl\`eme des groupes de congruence
pour $\mathrm{SL}_2,$} Ann.\:Math.\,{\bf 92}(1970), 489-527.

\bibitem{Serre2} J-P.~Serre, {\it Galois Cohomology}, Springer, 1997.

\bibitem{Spr} T.A.~Springer, {\it Linear Algebraic Groups}, 2nd edition, Birkh\"auser, 1998.


\bibitem{Stein} M.R.~Stein, {\it Surjective stability in dimension $0$ for $K_2$ and
related functors},  Trans. AMS {\bf 178}(1973), 165-191.


\bibitem{Sw} R.G.~Swan, {\it Generators and relations for certain special linear groups,} Adv. Math.
{\bf 6}(1971), 1-77.

\bibitem{Sury} B.~Sury, {\it Congruence subgroup problem for anisotropic groups
over semilocal rings}, Proc. Indian Acad. Sci. Math. Sci. {\bf 101}(1991), 87-110.

\bibitem{T-simple} J.~Tits, {\it Algebraic and abstract simple groups,} Ann. Math. {\bf 80}(1964), 313-329.


\bibitem{T} J.~Tits, {\it Classification of algebraic semisimple
groups,} in: Algebraic Groups and Discontinuous Groups, in: Proc.
Sympos. Pure Math., vol. 9, Amer.\,Math.\,Soc., 1966, pp. 33-62.


\bibitem{Tits-Bour} J.~Tits, {\it Groupes de Whitehead de groupes alg\'ebriques simples sur un corps (d'apr\`es V.P.~Platonov et al)},
S\'eminaire Bourbaki, 1976/1977, exp. 505, pp. 218-236


\bibitem{To1} G.~Tomanov, {\it On the congruence subgroup problem for some anisotropic algebraic groups over number fields}, J. Reine Angew. Math.
{\bf 402}(1989), 138-152.



\bibitem{To2} G.~Tomanov, {\it Congruence subgroup problem for groups of type $G_2$}, C.R. Acad. Bulgare Sci. {\bf 42}(1989), 9-11.

\bibitem{Tur} G.~Turnwald, {\it Multiplicative subgroups of finite index in a division ring}, Proc. AMS {\bf 120}(1994), 377-381.


\bibitem{Wa} N.R.~Wallach, {\it Square integrable automorphic
forms and cohomology of arithmetic quotients of $SU(p , q),$} Math.
Ann. {\bf 266}(1984), 261-278.


\bibitem{Z1} P.A.~Zalesskii, {\it Normal subgroups of free
constructions of profinite groups, and the congruence kernel in the
case of positive characteristic,} Izv. math. {\bf 59}(1995),
499-516.

\bibitem{Z2} P.A.~Zalesskii, {\it Profinite surface groups and
the congruence kernel of arithmetic lattices in
$\mathrm{SL}_2(\R)$},  Israel J. Math. {\bf 146}(2005), 111-123.

\bibitem{Zel1} E.I.~Zel'manov, {\it Solution of the restricted Burnside problem for groups of odd
exponent}, Math. USSR-Izv. {\bf 36}(1991), No. 1, 41-60.


\bibitem{Zel2} E.I.~Zel'manov, {\it Solution of the restricted Burnside problem for 2-groups}, Math.
USSR-Sb. {\bf 72}(1992), No. 2, 543-565.


\end{thebibliography}

\end{document}